\documentclass[times,sort&compress,3p]{elsarticle}
\journal{}
\usepackage[labelfont=bf]{caption}

\usepackage{amsmath,amsfonts,amssymb,amsthm,booktabs,color,epsfig,graphicx,hyperref,url,commath,amsthm,mathtools,thmtools}
\usepackage{bbm}
\usepackage[dvipsnames]{xcolor}
\usepackage[capitalise]{cleveref}
\usepackage{microtype}        
\usepackage{bbm}              

\declaretheorem[Refname={Theorem,Theorems}]{theorem}
\declaretheorem[style=definition,Refname={Definition,Definitions}]{definition}
\declaretheorem[style=definition,Refname={Assumption, Assumptions}]{assumption}

\declaretheorem[style=definition,Refname={Remark,Remarks}]{remark}
\declaretheorem[Refname={Lemma,Lemmas}]{lemma}

\declaretheorem[name=Proposition,Refname={Proposition,Propositions}]{proposition}

\let\oldproofname=\proofname
\renewcommand{\proofname}{\rm\bf{\oldproofname}}

\renewcommand{\b}[1]{\mathbf{#1}} 

\DeclarePairedDelimiterX\Set[2]{\lbrace}{\rbrace}%
{ #1 \,:\, #2 }                                         
\DeclarePairedDelimiterX\inprod[2]{\langle}{\rangle}%
{ #1 , #2 }                                             

\newcommand{\R}{\mathbb{R}} 
\newcommand{\N}{\mathbb{N}} 

\newcommand{\GP}{\textup{GP}}
\newcommand{\FE}{\textup{FE}}

\newcommand{\T}{\intercal} 

\makeatletter
\def\ps@pprintTitle{%
 \let\@oddhead\@empty
 \let\@evenhead\@empty
 \def\@oddfoot{}%
 \let\@evenfoot\@oddfoot}
\makeatother

\begin{document}

\begin{frontmatter}

\title{Error analysis for a statistical finite element method}

\author[1,2]{Toni Karvonen\corref{mycorrespondingauthor}}
\author[3]{Fehmi Cirak}
\author[3]{Mark Girolami}

\address[1]{School of Engineering Sciences, Lappeenranta--Lahti University of Technology LUT, Lappeenranta, Finland}
\address[2]{Department of Mathematics and Statistics, University of Helsinki, Finland}
\address[3]{Department of Engineering, University of Cambridge, United Kingdom}

\cortext[mycorrespondingauthor]{Corresponding author. Email address: \url{toni.karvonen@lut.fi}}

\begin{abstract}
  The recently proposed statistical finite element (statFEM) approach synthesises measurement data with finite element models and allows for making predictions about the unknown true system response.
  We provide a probabilistic error analysis for a prototypical statFEM setup based on a Gaussian process prior under the assumption that the noisy measurement data are generated by a deterministic true system response function that satisfies a second-order elliptic partial differential equation for an unknown true source term.
  In certain cases, properties such as the smoothness of the source term may be misspecified by the Gaussian process model.
  The error estimates we derive are for the expectation with respect to the measurement noise of the $L^2$-norm of the difference between the true system response and the mean of the statFEM posterior.
  The estimates imply polynomial rates of convergence in the numbers of measurement points and finite element basis functions and depend on the Sobolev smoothness of the true source term and the Gaussian process model.
  A numerical example for Poisson's equation is used to illustrate these theoretical results.
\end{abstract}


\end{frontmatter}

\section{Introduction}

The finite element method has become an indispensable tool for solving partial differential equations in engineering and applied sciences.
Today, the design, manufacture and maintenance of most engineering products rely on mathematical models based on finite element discretised partial differential equations (PDEs).
These models depend on a wide range of parameters, including material, geometry, and loading, which are inevitably subject to both epistemic and aleatoric uncertainties.
Consequently, the response of the actual engineering product and the inevitably misspecified mathematical model often bear little resemblance to each other, resulting in inefficient designs and overtly cautious operational decisions.
Fortunately, more and more engineering products are equipped with sensor networks that provide operational measurement data~\citep[e.g.,][]{febrianto:2021}.
The statistical finite element method (statFEM) allows us to infer the true system response by synthesising limited measurement data with finite element models~\citep{Girolami2021}.
By adopting a Bayesian approach, the prior probability measure of the finite element solution is obtained from the misspecified finite element model by solving a probabilistic forward problem.
Although any parameters of the finite element model can be random, in this article only the source term of the respective PDE is random and Gaussian, so that the finite element solution is Gaussian.
The assumed data-generating process for determining the likelihood of the measured data is additively composed of the random finite element solution, the known random measurement noise, and, possibly, an unknown random discrepancy component.
The chosen prior and the likelihood ensure that the posterior finite element probability density conditioned on the measurement data is Gaussian and easily computable.

More concretely, we consider the following mathematical problem. 
Let $\mathcal{L}$ be a differential operator and suppose that the system response~$u$ that generated the measurement data is given by the solution of
\begin{equation} \label{eq:pde-intro}
  \mathcal{L} u = f
\end{equation}
on a bounded domain $\Omega$ with zero Dirichlet boundary conditions.
The statistical component of the statFEM solution arises from the placement of a stochastic process prior on the source term $f$ and, possibly, the differential operator $\mathcal{L}$ or some of its parameters.
Doing this induces a stochastic process prior over the solution $u$.
After hyperparameter estimation and inclusion of additional levels of statistical modelling~\citep{KennedyOHagan2002}, which may account for various modelling discrepancies, one uses Bayesian inference to obtain a posterior distribution over the PDE solution given the measurement data.
The posterior can then be used predict the system behaviour at previously unseen data locations and  provide associated uncertainty quantification.
See~\citep{Girolami2021, Duffin2021, koh2023stochastic} for applications of this methodology to different types of PDEs and~\citep{Abdulle2021} for a somewhat different approach focusing on random meshes.
In any non-trivial setting, computation of the prior for $u$ from that placed on $f$ requires solving the PDE~\eqref{eq:pde-intro}.
In statFEM, the PDE is solved using finite elements.
Due to their tractability, Gaussian processes (GPs) are often the method of choice for modelling physical phenomena.
In the PDE setting that we consider Gaussian processes are particularly convenient because a GP prior, $f_\GP$, on $f$ induces a GP prior, $u_\GP$, on $u$ if the PDE is linear (see Fig.\@~\ref{fig:intro-gp}).
The induced prior $u_\GP$ has been studied in~\citep{Owhadi2015, Raissi2017} and Section~3.1.2 of~\citep{Cockayne2017}.
Although $u_\GP$ is generally not available in closed form, it is straightforward to approximate its mean and covariance functions from those of $f_\GP$ by using finite elements.

In this article, we provide estimates of the predictive error for the GP-based statFEM when the data are noisy evaluations of some deterministic true system response function $u_t$ which is assumed to be the solution of~\eqref{eq:pde-intro} for an unknown---but deterministic---true source term $f_t$.
Due to the complexity and difficulty of analysing a full statFEM approach, we consider a prototypical version that consists of a GP prior on $f$ and, possibly, a GP discrepancy term.
Scaling and other parameters these processes may have are assumed fixed.
Despite recent advances in understanding the behaviour of GP hyperparameters and their effect on the convergence of GP approximation~\citep{Karvonen2020, Teckentrup2020, Wynne2020}, these results are either not directly applicable in our setting or too generic, in that they assume that the parameter estimates remain in some compact sets, which has not been verified for commonly used parameter estimation methods, such as maximum likelihood.

As mentioned, finite elements are needed for computation of the induced prior $u_\GP$ and the associated posterior.
But why not simply use a readily available and explicit GP prior for $u$, such as one with a Matérn covariance kernel, instead of something that requires finite element approximations?
The main reason (besides this being the first step towards analysing the full statFEM) is that a prior $u_\GP$, for which  $\mathcal{L} u_\GP = f_\GP$, satisfies the structural constraints imposed by the PDE model and can therefore be expected to yield more accurate point estimates and more reliable uncertainty quantification than a more arbitrary prior if the data are generated by a solution of~\eqref{eq:pde-intro} for some source term.
We give a detailed description of the considered method in Section~\ref{sec:statfem}.

\begin{figure}

  \centering
  \includegraphics[scale=1.1]{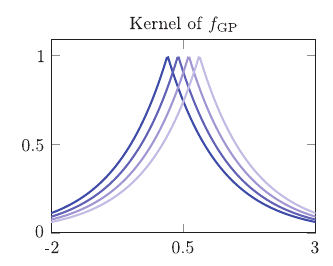} \hspace{0.05\textwidth}
  \includegraphics[scale=1.1]{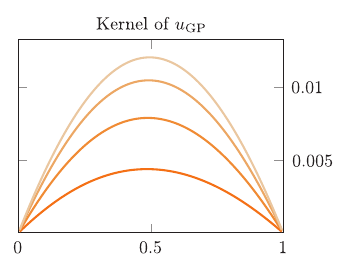}  
  \caption{Left: Four translates $K(\cdot, x)$ of the Matérn covariance kernel in~\eqref{eq:matern} with $\nu = 1/2$ and $\sigma = \ell = 1$. Right: Four translates $K_u(\cdot, x)$ of the corresponding kernel of the process $u_\GP$ for Poisson's equation with zero boundary conditions.} \label{fig:intro-gp}
  
\end{figure}

\subsection{Contributions}

Our contribution consists of a number of error estimates for the statFEM approach sketched above. 
Suppose that the measurements are $y_i = u_t(\b{x}_i) + \varepsilon_i$ for $n$ locations $\b{x}_i \in \Omega \subset \R^d$ and independent Gaussian noises $\varepsilon_i \sim \mathrm{N}(0, \sigma_\varepsilon^2)$.
The regression error estimates we prove are of the form
  \begin{equation} \label{eq:error-estimate-intro}
    \mathbb{E}\big[ \norm[0]{ u_t - \widetilde{m}}_{L^2(\Omega)} \big] \leq C_1 n^{-1/2 + a} + C_2 n_\FE^{-q} n^{3/2},
  \end{equation}
  where $\widetilde{m}$ is a posterior mean function obtained from statFEM and the expectation is with respect to the measurement noise.
  The constant $a \in (0, 1/2)$ depends on the smoothness of $f_t$ and $q > 0$ is the dimension~$d$ dependent characteristic rate of convergence of the finite element approximation with~$n_\FE$ elements.
  In~\eqref{eq:error-estimate-intro} it is assumed that the points $\b{x}_i$ cover $\Omega$ sufficiently uniformly.
  In Section~\ref{sec:error-estimates} we present error estimates for four different variants of statFEM, each of which corresponds to a different $\widetilde{m}$:
  \begin{itemize}
  \item Theorem~\ref{theorem:error-exact} assumes that no finite element discretisation is required for computation of $\widetilde{m}$. In this case $C_1 > 0$ and $C_2 = 0$.  It is required that $f_t$ be at least as smooth as the prior $f_\GP$.
  \item In Theorem~\ref{theorem:error-fe}, the more realistic assumption that $\widetilde{m}$ is constructed via a finite element approximation is used. In this case $C_1, C_2 > 0$. It is required that $f_t$ be at least as smooth as the prior $f_\GP$.
  \item Theorems~\ref{thm:error-with-discrepancy} and~\ref{thm:error-fe-with-discrepancy} concern versions which include a GP discrepancy term $v_\GP$ (i.e., the prior for $u$ is $u_\GP + v_\GP$) and do not use or use, respectively, finite element discretisation to compute $\widetilde{m}$. These theorems allow the priors to misspecify the source term and system response smoothness as it is not required that $f_t$ be at least as smooth as~$f_\GP$ or that $u_t$ be at least as smooth as $v_\GP$ or $u_\GP$.
  Such misspecification results are valuable because one often lacks precise information about the smoothness (or other properties) of $u_t$.
  \end{itemize}
  As discussed in Remark~\ref{remark:optimal-rate}, these rates are likely slightly sub-optimal.
  Some numerical examples for one-dimensional Poisson's equation are given in Section~\ref{sec:examples}.

  The proofs of these results are based on reproducing kernel Hilbert space (RKHS) techniques which are commonly used to analyse approximation properties of GPs~\citep{VaartZanten2011, Cialenco2012, Cockayne2017, Karvonen2020, Teckentrup2020, WangTuoWu2020, Wynne2020}.
  Our central tool is Theorem~\ref{thm:Ku-RKHS}, which describes the RKHS associated to the prior $u_\GP$ under the assumptions that the RKHS for $f_\GP$ is a Sobolev space and $\mathcal{L}$ is a second-order elliptic differential operator.
  This result is used to ``export'' (a) regression error estimates in some of the aforementioned references and (b) bounds on the concentration function~\citep{LiLinde1999, VaartZanten2011} from a ``standard'' GP prior $f_\GP$ (e.g., one specified by a Matérn covariance kernel) to the transformed prior $u_\GP$.
  When a finite element approximation is used, the regression error estimates are combined with a simple result (Proposition~\ref{prop:kernel-error}) which bounds the difference between GP posterior means for two different kernels in terms of the maximal difference of the kernels.
  
\subsection{Related work}

Solving PDEs with kernel-based methods goes back at least to the work of \citet{Kansa1990}; see~\citep{Fasshauer1996, FrankeSchaback1998} as well as Chapter~16 in~\citep{Wendland2005} for a more general treatment.
In the language of GPs, this radial basis function collocation approach is essentially based on modelling $u$ as a GP with a given covariance kernel and conditioning on the derivative observations $\mathcal{L}u(\b{x}_i) = f(\b{x}_i)$.
Typically no synthesis of actual measurement data is present (though this could be easily included).
For convergence results in a well-specified setting, see for example Theorem~16.15 in~\citep{Wendland2005}.
In a GP setting similar methods have been proposed and analysed by \citet{Graepel2003, Cialenco2012, Cockayne2017}, and \citet{Raissi2017}.
For some error estimates, see Lemma~3.4 and Proposition~3.5 in~\citep{Cialenco2012}.
Priors and covariance kernels derived from Green's function have been considered by \citet{FasshauerYe2011, FasshauerYe2013} and \citet{Owhadi2015}.
Furthermore, \citet{Papandreou2021} have recently derived bounds on the Wasserstein distance $W_2$ between the ideal prior and posterior (see \Cref{sec:gp-inference} in the present article) and their finite element approximations.

\section{Statistical finite element methods} \label{sec:statfem}

This section describes the statFEM approach that is analysed in Section~\ref{sec:error-estimates} and discusses some extensions that are not covered by our analysis.
We begin by defining the class of second-order elliptic PDE problems that are considered in this article.

Let $d \in \N$ and $\Omega \subset \R^d$ be an open and bounded set which satisfies an interior cone condition~\citep[e.g.,][Definition~3.6]{Wendland2005} and has a Lipschitz boundary $\partial \Omega$ (i.e., the boundary is locally the graph of a Lipschitz function).
Occasionally we also require an assumption that $\partial \Omega$ be $C^k$ or $C^{k,\alpha}$, which means that its boundary can be interpreted locally as the graph of a function in $C^{k}(\R^{d-1})$ or in the Hölder space $C^{k,\alpha}(\R^d)$.
Definitions of Hölder spaces and the smoothness of the domain can be found in Sections~\ref{sec:rkhs} and~\ref{sec:pde-regularity}.

Let $\mathcal{L}$ be a second-order partial differential operator of the form
\begin{equation} \label{eq:differential-operator}
  \mathcal{L}u = - \sum_{i=1}^d \sum_{j = 1}^d \partial_i ( a_{ij} \partial_j u ) + \sum_{i=1}^d b_i \partial_i u + c u
\end{equation}
for coefficient functions $a_{ij}$, $b_i$ and $c$ which are bounded on the closure $\bar{\Omega}$.
We further assume that $a_{ij} \in C^1(\Omega)$ and $a_{ij} = a_{ji}$ for all $i,j \in \{ 1,\ldots,d\}$.
The differential operator is assumed to be uniformly elliptic, which is to say that there is a positive constant $\lambda$ such that \smash{$\sum_{i=1}^d \sum_{j=1}^d a_{ij}(\b{x}) z_i z_j \geq \lambda \norm[0]{\b{z}}_2^2$} for any $\b{x} \in \Omega$ and $\b{z} \in \R^d$.
Moreover, our results use the following regularity assumption.

\begin{assumption}[Regularity] \label{assumption:regularity}
  For a given $k \in \N_0$, the boundary $\partial \Omega$ is $C^{k+2}$ and $a_{ij}, b_i, c \in C^{k+1}(\bar{\Omega})$ for all $i,j \in \{1,\ldots,d\}$.
\end{assumption}

Given a known differential operator $\mathcal{L}$, we consider the elliptic PDE
\begin{equation} \label{eq:elliptic-pde}
  \begin{alignedat}{3}
  \mathcal{L} u &= f  \qquad && \text{in } \: \Omega, \\
  u &= 0  && \text{on } \: \partial \Omega
  \end{alignedat}
\end{equation}
for a source term $f \colon \Omega \to \R$.
Let $\mathcal{H}_\mathcal{L}(\Omega)$ be some space of functions defined on $\Omega$ such that the above PDE admits a unique classical (i.e., pointwise) solution $u \colon \bar{\Omega} \to \R$ for every $f \in \mathcal{H}_\mathcal{L}(\Omega)$.
Therefore there is a linear operator $\mathcal{L}^{-1} \colon \mathcal{H}_\mathcal{L}(\Omega) \to C^2(\Omega)$ such that $u = \mathcal{L}^{-1} f$ is the unique solution of~\eqref{eq:elliptic-pde} for any $f \in \mathcal{H}_\mathcal{L}(\Omega)$.
Suppose that there is a true system response $u_t$ which is the unique solution of
\begin{equation} \label{eq:true-system-response}
  \begin{alignedat}{3}
  \mathcal{L} u_t &= f_t  \qquad && \text{in } \: \Omega, \\
  u_t &= 0  && \text{on } \: \partial \Omega
  \end{alignedat}
\end{equation}
for a certain true source term $f_t \in \mathcal{H}_\mathcal{L}(\Omega)$, which may be unknown, and that one has access to $n \in \N$ noisy observations $\b{Y} = (y_1, \ldots, y_n) \in \R^n$ of $u_t$ at distinct data locations $X = \{ \b{x}_1, \ldots, \b{x}_n \} \subset \Omega$: for $i \in \{1, \ldots, n\}$,
\begin{equation} \label{eq:observations}
  y_i = u_t(\b{x}_i) + \varepsilon_i
\end{equation}
where $\sigma_\varepsilon > 0$ and $\varepsilon_i \sim \mathrm{N}(0, \sigma_\varepsilon^2)$ are i.i.d.
The statFEM approach provides a means for predicting the value of the true system response, $u_t(\b{x})$, at any point, $\b{x}$, in the domain.
It also provides uncertainty estimates for these predictions that are based on the observations, the differential operator $\mathcal{L}$, and a prior which encodes, for example, assumptions on the smoothness of the true source term.
We emphasise that here $u_t$ and $f_t$ are always assumed to be some fixed and deterministic functions.
Although we use GPs to model them, the functions themselves are not considered stochastic processes.
All expectations that occur in this article are with respect to the noise variables $\varepsilon_i$ alone.
The locations $\b{x}_i$ are not assumed random in this article.

\subsection{Gaussian process inference} \label{sec:gp-inference}

Let $K \colon \Omega \times \Omega \to \R$ be a positive-semidefinite kernel. That is, for any $n \in \N$, $\alpha_1, \ldots, \alpha_n \in \R$, and $\b{z}_1, \ldots, \b{z}_n \in \Omega$ it holds that $\sum_{i=1}^n \sum_{j=1}^n \alpha_i \alpha_j K(\b{z}_i, \b{z}_j) \geq 0$.
One of the most ubiquitous---as well as one which appears repeatedly in this article---classes of positive-semidefinite kernels is that of the Matérn kernels
\begin{equation} \label{eq:matern}
  K(\b{x}, \b{y}) = \sigma^2 \, \frac{2^{1-\nu}}{\Gamma(\nu)} \bigg( \frac{\sqrt{2\nu} \norm[0]{\b{x}-\b{y}}_2}{\ell} \bigg)^\nu \mathrm{K}_\nu \bigg( \frac{\sqrt{2\nu} \norm[0]{\b{x}-\b{y}}_2}{\ell} \bigg),
\end{equation}
where $\nu > 0$ is a smoothness parameter, $\ell > 0$ a length-scale parameter, $\sigma > 0$ a scaling parameter, $\Gamma$ the gamma function, and $\mathrm{K}_\nu$ the modified Bessel function of the second kind of order $\nu$.
These kernels are important because they induce Sobolev spaces (see Section~\ref{sec:rkhs}).
We model $f_t$ as a Gaussian process $f_\GP \sim \mathrm{GP}(m, K)$ and assume that
\begin{equation} \label{eq:kernel-HL-assumption}
  \text{(i) } m \in \mathcal{H}_\mathcal{L}(\Omega), \quad \text{(ii) } K(\cdot, \b{x}) \in \mathcal{H}_\mathcal{L}(\Omega), \quad \mathcal{L}_{\b{x}}^{-1} K(\cdot, \b{x}) \in \mathcal{H}_\mathcal{L}(\Omega), \quad \b{x} \in \Omega,
\end{equation}
where the subscript denotes the variable with respect to which the linear operator is applied.
These assumptions ensure that various functions that we are about to introduce are unique and pointwise well-defined.
Because $\mathcal{L}$ is a linear differential operator, the above GP prior over $f_t$ induces the prior \mbox{$u_\GP \sim \mathrm{GP}(m_u, K_u)$} over $u_t$ with mean function $m_u = \mathcal{L}^{-1} m$ and covariance kernel $K_u$, which satisfies
\begin{equation} \label{eq:Ku-definition}
  \mathcal{L}_{\b{x}} \mathcal{L}_{\b{y}} K_u( \b{x}, \b{y} ) = K(\b{x}, \b{y})
\end{equation}
for all $\b{x}, \b{y} \in \Omega$, as well as $\mathcal{L} K_u(\cdot, \b{y}) = 0$ on $\partial \Omega$ for every $\b{y} \in \bar{\Omega}$.
The existence and uniqueness of the mean and covariance are guaranteed by the assumptions in~\eqref{eq:kernel-HL-assumption}.

To arrive at an ideal version of the GP-based statFEM we condition the GP $u_\GP$ on the measurement data in~\eqref{eq:observations}.
This yields the conditional process 
\begin{equation*}
  u_\GP \mid \b{Y} \sim \mathrm{GP}( m_{u \mid \b{Y}}, K_{u \mid \b{Y}}) ,
\end{equation*}
whose mean and covariance are 
\begin{subequations} \label{eq:exact-conditional-moments}
  \begin{align}
  m_{u \mid \b{Y}}(\b{x}) &= m_u(\b{x}) + \b{K}_u(\b{x}, X)^\T \big( \b{K}_u(X, X) + \sigma_\varepsilon^2 \b{I}_n \big)^{-1} ( \b{Y} - \b{m}_u(X) ), \label{eq:exact-conditional-mean} \\
  K_{u \mid \b{Y}}(\b{x}, \b{y}) &= K_u(\b{x}, \b{y}) - \b{K}_u(\b{x}, X)^\T \big( \b{K}_u(X, X) + \sigma_\varepsilon^2 \b{I}_n \big)^{-1} \b{K}_u(\b{y}, X),
  \end{align}
\end{subequations}
where $\b{K}_u(X, X)$ is the $n \times n$ kernel matrix with elements $K_u(\b{x}_i, \b{x}_j)$, $\b{K}_u(\b{x}, X)$ and $\b{m}_u(X)$ are $n$-vectors with elements $K_u(\b{x}, \b{x}_i)$ and $m_u(\b{x}_i)$, respectively, and $\b{I}_n$ is the $n \times n$ identity matrix.
However, the mean function $m_u = \mathcal{L}^{-1} m$ and the covariance kernel $K_u$ in~\eqref{eq:Ku-definition} cannot be solved in closed form in all but the simplest of cases.
This necessitates replacing their occurrences in~\eqref{eq:exact-conditional-moments} with finite element approximations.

\subsection{Finite Element methods} \label{sec:fe}

Let $H_0^1(\Omega)$ denote the Sobolev space of order one and zero trace. 
Informally, this space consists of once weakly differentiable functions that vanish on $\partial \Omega$.
A function $u \in H_0^1(\Omega)$ is a weak solution to the PDE in~\eqref{eq:elliptic-pde} if
\begin{equation} \label{eq:weak-formulation}
  B(u, v) = \int_\Omega f(\b{x}) v(\b{x}) \dif \b{x}, \quad v \in H_0^1(\Omega) ,
\end{equation}
where
\begin{equation*}
  B(u, v) = \int_\Omega \bigg( \sum_{i=1}^d \sum_{j = 1}^d a_{ij}(\b{x}) [ \partial_i u(\b{x}) ] [ \partial_j v(\b{x}) ] + \sum_{i=1}^d b_i(\b{x}) [ \partial_i u(\b{x}) ] v(\b{x}) + c(\b{x}) u(\b{x}) v(\b{x}) \bigg) \dif \b{x}
\end{equation*}
is the bilinear form associated with the elliptic differential operator $\mathcal{L}$ in~\eqref{eq:differential-operator}.
The bilinear form is obtained from $\int_\Omega [\mathcal{L} u(\b{x})] v(\b{x}) \dif \b{x} = \int_\Omega f(\b{x}) v(\b{x}) \dif \b{x}$ via integration by parts.

Finite element methods construct approximations to the weak solution by replacing the infinite-dimensional Sobolev space in~\eqref{eq:weak-formulation} with a finite-dimensional trial space $V \subset H_0^1(\Omega)$ and finding $u^\FE \approx u$ in $V$ such that 
\begin{equation*}
  B(u^\FE, v) = \int_\Omega f(\b{x}) v(\b{x}) \dif \b{x}, \quad v \in V.
\end{equation*}
If $V$ is an $n_\FE$-dimensional vector space with basis \smash{$\phi_1, \ldots, \phi_{n_\FE}$}, it follows from the bilinearity of $B$ and linearity of integration that 
\begin{equation} \label{eq:fe-approximation}
  u^\FE(\b{x}) = \sum_{i=1}^{n_\FE} u_i \phi_i(\b{x}),
\end{equation}
where the coefficient vector $\b{u} = (u_1, \ldots, u_{n_\FE})$ is the solution of the linear system
\begin{equation} \label{eq:fe-coef-eq}
  \begin{pmatrix}
    B(\phi_1, \phi_1) & \cdots & B(\phi_{n_\FE}, \phi_1) \\
    \vdots & \ddots & \vdots \\
    B(\phi_{n_\FE}, \phi_1) & \cdots & B(\phi_{n_\FE}, \phi_{n_\FE})
  \end{pmatrix}
  \begin{pmatrix}
    u_1 \\ \vdots \\ u_{n_\FE}
  \end{pmatrix}
  =
  \begin{pmatrix}
    \int_\Omega f(\b{x}) \phi_1(\b{x}) \dif \b{x} \\ \vdots \\ \int_\Omega f(\b{x}) \phi_{n_\FE}(\b{x}) \dif \b{x}
  \end{pmatrix},
\end{equation}
where the $n_\FE \times n_\FE$ matrix on the left-hand side is called the stiffness matrix and denoted by $\b{A}$.

The choice of the trial space $V \subset H_0^1(\Omega)$ determines how well $u^\FE$ approximates $u$.
The simplest and most prevalent approach is to use piecewise linear basis functions. 
Assume for a moment that $\Omega$ is a polyhedron.
One first constructs a triangular partition of $\Omega$ (or mesh) that consists of simplices $\Omega_1, \ldots, \Omega_{\tilde{n}}$ defined by $n_\FE < \tilde{n}$ vertices $\b{x}_1^\FE, \ldots, \b{x}_{n_\FE}^\FE$ in the interior of $\Omega$ plus additional vertices on the boundary and then selects $V$ as the space of continuous functions that are piecewise linear in each simplex and zero on the boundary. 
This space is spanned by piecewise linear hat functions $\phi_1, \ldots, \phi_{n_\FE}$ with the property $\phi_i(\b{x}_j^\FE) = \delta_{ij}$.
The finite element approximation is then given by~\eqref{eq:fe-approximation}, where the coefficients $u_i$ are obtained from the linear system~\eqref{eq:fe-coef-eq}.
See~\citep{BrennerScott2008, lord2014introduction, KnabnerAngermann2021} for detailed reviews of finite element methods.
Fig.\@~\ref{fig:finite-elements} shows a triangular mesh for a two-dimensional domain and piecewise linear hat functions.

\begin{figure}
  \centering
  \includegraphics[width=0.30\textwidth]
  {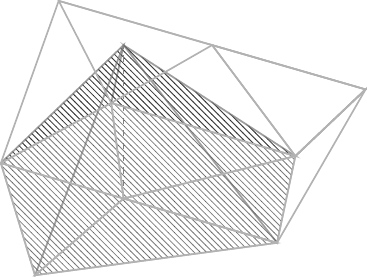}
  \includegraphics[width=0.30\textwidth]
  {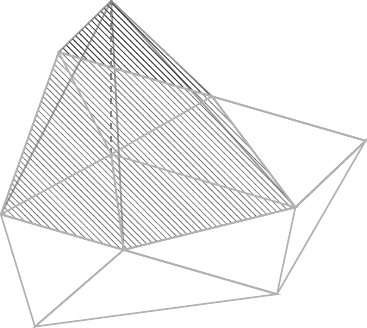}
  \includegraphics[width=0.30\textwidth]
  {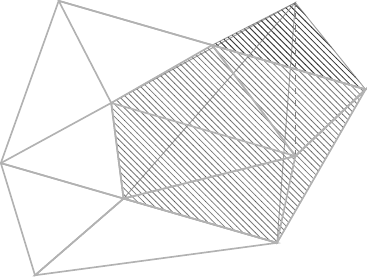}
  \caption{Triangulation of a polyhedral domain in $\R^2$ and three piecewise linear hat functions that vanish on the boundary.}
  \label{fig:finite-elements}
\end{figure}

\subsection{Finite element posterior}

One can approximate the mean $m_u$ and covariance $K_u$ of $u_\GP \sim \GP(m_u, K_u)$ with the finite element method.
Because $\mathcal{L} m_u = m$, by solving $\b{u}$ from~\eqref{eq:fe-coef-eq} with $f = m$ we obtain the mean approximation
\begin{equation*}
  m_u^\FE(\b{x}) = \pmb{\mu}^\T \b{A}^{-1} \pmb{\phi}(\b{x}) \approx m_u(\b{x}),
\end{equation*}
where $\pmb{\phi}(\b{x}) = (\phi_1(\b{x}), \ldots, \phi_{n_\FE}(\b{x}))$ and $\pmb{\mu} = (\int_\Omega m(\b{x}) \phi_1(\b{x}) \dif \b{x}, \ldots, \int_\Omega m(\b{x}) \phi_{n_\FE}(\b{x}) \dif \b{x})$.
Because $\mathcal{L}_{\b{x}} \mathcal{L}_{\b{y}} K_u( \b{x}, \b{y} ) = K(\b{x}, \b{y})$, we may approximate $K_u$ by first forming an approximation with $f = K(\cdot, \b{y})$ in~\eqref{eq:fe-coef-eq} and subsequently forming a second approximation with the first approximation as $f$ in~\eqref{eq:fe-coef-eq}.
From this we obtain the covariance approximation
\begin{equation*}
  K_u^\FE(\b{x}, \b{y}) = \pmb{\phi}(\b{x})^\T \b{A}^{-1} \b{M} \b{A}^{-1}\pmb{\phi}(\b{y}) \approx K_u(\b{x}, \b{y}),
\end{equation*}
where the matrix $\b{M} \in \R^{n_\FE \times n_\FE}$ has the elements
\begin{equation*}
  (\b{M})_{ij} = \int_\Omega \int_\Omega \phi_i(\b{x}') K(\b{x}', \b{y}') \phi_j(\b{y}') \dif \b{x}' \dif \b{y}'.
\end{equation*}
Substituting $m_u^\FE$ for $m_u$ and $K_u^\FE$ for $K_u$ yields the approximations
\begin{subequations} \label{eq:fe-conditional-moments}
  \begin{align}
  m_{u \mid \b{Y}}^\FE(\b{x}) &= m_u^\FE(\b{x}) + \b{K}_u^\FE(\b{x}, X)^\T \big( \b{K}_u^\FE(X, X) + \sigma_\varepsilon^2 \b{I}_n \big)^{-1} ( \b{Y} - \b{m}_u^\FE(X) ), \label{eq:fe-conditional-mean} \\
  K_{u \mid \b{Y}}^\FE(\b{x}, \b{y}) &= K_u(\b{x}, \b{y}) - \b{K}_u^\FE(\b{x}, X)^\T \big( \b{K}_u^\FE(X, X) + \sigma_\varepsilon^2 \b{I}_n \big)^{-1} \b{K}_u^\FE(\b{y}, X)
  \end{align}
\end{subequations}
to the conditional moments in~\eqref{eq:exact-conditional-moments}.
In practice, it may be tedious or impossible to compute the elements of $\b{M}$ in closed form.
When the supports of $\phi_i$ are contained within small neighbourhoods of some nodes $\b{x}_i^\FE \in \Omega$, one may treat the kernel as constant within these supports and employ the approximations
\begin{equation} \label{eq:kernel-integral-approximation}
  \int_\Omega \int_\Omega \phi_i(\b{x}') K(\b{x}', \b{y}') \phi_j(\b{y}') \dif \b{x}' \dif \b{y}' \approx \bigg( \int_\Omega \phi_i(\b{x}') \dif \b{x}' \bigg) K(\b{x}_i^\FE, \b{x}_j^\FE) \bigg( \int_\Omega \phi_i(\b{y}') \dif \b{y}' \bigg).
\end{equation}
The structure of the resulting statFEM approach is displayed in Fig.\@~\ref{fig:statfe}.

\begin{figure}
  \centering
  \includegraphics[width=0.6\textwidth]{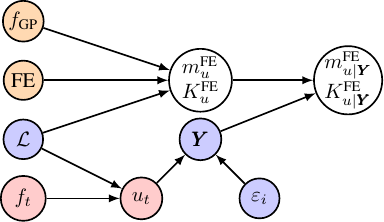}
  \caption{Graphical model of the statFEM approach considered in this article. Blue variables are known to the user while red variables are unknown. Orange variables represent parameters that the user can choose.} \label{fig:statfe}
\end{figure}

Later we will use the following generic assumption on the error of the finite element approximations.

\begin{assumption}[Finite element error] \label{assumption:fe-error}
  Let $q_1, q_2 > 0$ and $q = \min\{q_1, q_2\}$.
  There exist positive constants $C_1$ and $C_2$, which do not depend on $n_\FE$, such that
  \begin{equation*}
    \sup_{\b{x} \in \Omega} \abs[0]{ m_u(\b{x}) - m_u^\FE(\b{x}) } \leq C_1 n_\FE^{-q_1}, \quad \sup_{ \b{x}, \b{y} \in \Omega} \abs[0]{ K_u(\b{x}, \b{y}) - K_u^\FE(\b{x}, \b{y}) } \leq C_2 n_\FE^{-q_2}
  \end{equation*}
  for all $n_\FE \in \N$.
\end{assumption}

One should expect that $q_2 \leq q_1$ because the approximation to the covariance $K_u$ is obtained via two-fold finite element approximation.
In the error analysis of finite element methods it is typical to express error estimates in terms of a characteristic simplex size defined as 
\begin{equation*}
  h = \max_{i \in \{1,\ldots,\tilde{n}\}} \, \min_{\b{x}, \b{y} \in \Omega_i} \lVert \b{x} - \b{y} \rVert_2
\end{equation*}
instead of the number of elements as we have done in \Cref{assumption:fe-error}.
If the mesh is quasi-uniform, in that $h$ tends to zero as \smash{$n_\FE^{-1/d}$}, rates expressed in terms of $h$ are easily converted to rates in terms of $n_\FE$.
Our focus is on the statistical component of the statFEM approach and we do not wish to introduce the machinery that is necessary for precise presentation of estimates on the error of finite elements in the uniform norm.
Under certain regularity assumptions one can show that piecewise linear finite elements (see Section~\ref{sec:fe} and Fig.\@~\ref{fig:finite-elements}) for Poisson's equation $-\Delta u = f$ in dimension $d$ converge with the rate \smash{$h^2 \lvert \ln h \rvert$} in the uniform norm, which translates to the rate \smash{$n_\FE^{-2/d} \ln n$} in the quasi-uniform case~\citep{Nitsche1977}.
See~\citep{SchatzWahlbin1977, SchatzWahlbin1978, SchatzWahlbin1995} for additional estimates in the uniform norm.
For more casual introductions to error estimates we refer to \citep{BrennerScott2008, lord2014introduction}.
An additional motivation for \Cref{assumption:fe-error} is the presence of integral approximations, such as that in~\eqref{eq:kernel-integral-approximation}, which are not incorporated into many of the existing error estimates for finite elements.

\subsection{The discrepancy term} \label{sec:discrepancy}

It is often desirable to include a discrepancy term $v_\GP \sim \mathrm{GP}(m_d, K_d)$ to account for modelling errors.
We do this by replacing the induced GP model $u_\GP \sim \mathrm{GP}(m_u, K_u)$ for $u_t$ with $u_\GP + v_\GP$, so that the full GP model for $u_t$ is
\begin{equation*}
  u_\GP + v_\GP \sim \mathrm{GP}(m_u + m_d, K_u + K_d).
\end{equation*}
Unlike $u_\GP$, which is induced by the GP prior $f_\GP$ over $f_t$ and is thus accessible only by solving the PDE~\eqref{eq:elliptic-pde}, the discrepancy term is typically taken to be a GP with some standard covariance kernel, such as a Matérn in~\eqref{eq:matern}.
Denote $m_{ud} = m_u + m_d$ and $K_{ud} = K_u + K_d$.
When the discrepancy term is included, the exact conditional moments in~\eqref{eq:exact-conditional-moments} become
\begin{subequations} \label{eq:exact-conditional-moments-discrepancy}
  \begin{align}
  m_{d; u \mid \b{Y}}(\b{x}) &= m_{ud}(\b{x}) + \b{K}_{ud}(\b{x}, X)^\T \big( \b{K}_{ud}(X, X) + \sigma_\varepsilon^2 \b{I}_n \big)^{-1} ( \b{Y} - \b{m}_{ud}(X)), \label{eq:exact-conditional-mean-discrepancy} \\
  K_{d; u \mid \b{Y}}(\b{x}, \b{y}) &= K_{ud}(\b{x}, \b{y}) - \b{K}_{ud}(\b{x}, X)^\T \big( \b{K}_{ud}(X, X) + \sigma_\varepsilon^2 \b{I}_n\big)^{-1} \b{K}_{ud}(\b{y}, X).
  \end{align}
\end{subequations}
When a finite element approximation is employed, we get
\begin{subequations} \label{eq:fe-conditional-moments-discrepancy}
  \begin{align}
  m_{d; u \mid \b{Y}}^\FE(\b{x}) &= m_{ud}^\FE(\b{x}) + \b{K}_{ud}^\FE(\b{x}, X)^\T \big( \b{K}_{ud}^\FE(X, X) + \sigma_\varepsilon^2 \b{I}_n \big)^{-1} ( \b{Y} - \b{m}_{ud}^\FE(X)), \label{eq:fe-conditional-mean-discrepancy} \\
  K_{d; u \mid \b{Y}}^\FE(\b{x}, \b{y}) &= K_{ud}^\FE(\b{x}, \b{y}) - \b{K}_{ud}^\FE(\b{x}, X)^\T \big( \b{K}_{ud}^\FE(X, X) + \sigma_\varepsilon^2 \b{I}_n \big)^{-1} \b{K}_{ud}^\FE(\b{y}, X).
  \end{align}
\end{subequations}
where $m_{ud}^\FE = m_u^\FE + m_d$ and $K_{ud}^\FE = K_u^\FE + K_d$.

\subsection{Extensions}

In practice, a variety of additional levels of statistical modelling, or altogether a more complex PDE model, are typically used in statFEM~\citep{Girolami2021}.
These can include an additional factor with GP models placed on them on the left-hand side of~\eqref{eq:elliptic-pde}.
The standard example being is Poisson's equation 
\begin{equation} \label{eq:poisson-more-complicated}
  - \nabla ( \mathrm{e}^\mu \nabla u ) = f
\end{equation}
in which a GP prior is placed on $\mu$ (and the exponential ensures positivity of the diffusion coefficient, $\mathrm{e}^\mu$) in addition to $f$.
Estimation of various parameters present in model, such as parameters of the covariance kernel $K$ (e.g., $\sigma$, $\ell$ and $\nu$ of a Matérn kernel), is also common.
If GP priors are placed on $\mu$ and $f$ in the model~\eqref{eq:poisson-more-complicated} or its generalisation of some form, the prior induced on $u$ is no longer a GP.
This will render most of the theoretical tools that we use inoperative, and this generalisation is not accordingly pursued here.
While there is some recent theoretical work on parameter estimation in Gaussian process regression for deterministic data-generating functions and its effect on posterior rates of convergence and reliability~\citep{Karvonen2020, Teckentrup2020, Wang2021, Karvonen2023}, the results that have been obtained are not yet sufficiently general or flexible to be useful in the PDE setting considered here. 
Reliability refers to the quality of the posterior standard deviation as a measure of approximation error.
Informally speaking a GP model is reliable if the posterior standard deviation decays with the same rate as the approximation error as $n \to \infty$.
See~\citep{Szabo2015} for some work on reliability and its connection to model parameter estimation in the context of the Gaussian sequence model.

\section{Main Results} \label{sec:results}

This section contains the main results of this article.
The results provide rates of contraction, as $n \to \infty$, of the expectation (with respect to the observation noise distribution) of the $L^2(\Omega)$-norm between the true source term $u_t$ and the GP conditional means in~\eqref{eq:exact-conditional-mean},~\eqref{eq:fe-conditional-mean}, and~\eqref{eq:fe-conditional-mean-discrepancy}.
All proofs are deferred to Section~\ref{sec:proofs}.
The results are expressed in terms of the fill-distance
\begin{equation} \label{eq:fill-distance}
  h_{X,\Omega} = \sup_{ \b{x} \in \Omega} \, \min_{ i \in \{1,\ldots,n\}} \norm[0]{\b{x} - \b{x}_i}_2
\end{equation}
of the set of points $X \subset \Omega$.
The fill-distance cannot tend to zero with a rate faster than $n^{-1/d}$, a rate which is achieved by, for example, uniform Cartesian grids.

\subsection{Function spaces} \label{sec:rkhs}

Let $\mathrm{D}^{\b{\alpha}} f$ denote the weak derivative of order $\b{\alpha} \in \N_0^d$ of any sufficiently regular function $f \colon \Omega \to \R$.
Let $k \in \N_0$. The Sobolev space $H^k(\Omega)$ consists of functions for which $\mathrm{D}^{\b{\alpha}} f$ exists for all $\abs[0]{\b{\alpha}} \leq k$ and the norm
\begin{equation*}
  \norm[0]{f}_{H^k(\Omega)} = \bigg( \sum_{ \abs[0]{\b{\alpha}} \leq k } \norm[0]{ \mathrm{D}^{\b{\alpha}} f }_{L^2(\Omega)}^2 \bigg)^{1/2}
\end{equation*}
is finite.
The Hölder space $C^{k,\alpha}(\Omega)$ consists of functions which are $k \in \N_0$ times differentiable on $\Omega$ and whose derivatives of order $k$ are Hölder continuous with exponent $\alpha \in (0, 1]$.

Some of our assumptions are expressed in terms of reproducing kernel Hilbert spaces (RKHSs).
By the classical Moore--Aronszajn theorem~\citep[p.\@~19]{Berlinet2004} every positive-semidefinite kernel $K \colon \Omega \times \Omega \to \R$ induces a unique RKHS, $\mathcal{H}(K, \Omega)$, which consists of functions $f \colon \Omega \to \R$ and is equipped with an inner product $\inprod{\cdot}{\cdot}_K$ and norm $\norm[0]{\cdot}_K$.
Two fundamental properties of this space are that (i) $K(\cdot, \b{x})$ is an element of $\mathcal{H}(K, \Omega)$ for every $\b{x} \in \Omega$ and (ii) the reproducing property
\begin{equation} \label{eq:reproducing-property}
  \inprod{f}{K(\cdot, \b{x})}_K = f(\b{x}), \quad f \in \mathcal{H}(K, \Omega) \: \text{ and } \: \b{x} \in \Omega.
\end{equation}
Our results use an assumption that $\mathcal{H}(K, \Omega)$ is norm-equivalent to a Sobolev space.

\begin{definition}[Norm-equivalence] \label{assumption:norm-equivalence}
  The RKHS $\mathcal{H}(K, \Omega)$ is norm-equivalent to the Sobolev space $H^k(\Omega)$, denoted $\mathcal{H}(K, \Omega) \simeq H^k(\Omega)$, if $\mathcal{H}(K, \Omega) = H^k(\Omega)$ as sets and there exist positive constants $C_K$ and $C_K'$ such that
\begin{equation} \label{eq:norm-equivalence}
  C_K \norm[0]{f}_{H^k(\Omega)} \leq \norm[0]{f}_{K} \leq C_K' \norm[0]{f}_{H^k(\Omega)}
\end{equation}
for all $f \in \mathcal{H}(K, \Omega)$.
\end{definition}

The RKHS of a Matérn kernel of smoothness $\nu$ in~\eqref{eq:matern} is norm-equivalent to $H^{\nu+d/2}(\Omega)$.
If $k > d/2$, the Sobolev embedding theorem ensures that any kernel which is norm-equivalent to $H^k(\Omega)$ is continuous and that all functions in its RKHS are continuous.
From now on we assume that $\mathcal{H}(K, \Omega) \subset \mathcal{H}_\mathcal{L}(\Omega)$, which is to say that the PDE in~\eqref{eq:elliptic-pde} admits a unique classical solution for every $f \in \mathcal{H}(K, \Omega)$.

\subsection{Exact posterior} \label{eq:exact-error-estimates}

Our first result concerns an ideal statFEM that uses no finite element discretisation.
The relevant posterior moments are given in~\eqref{eq:exact-conditional-moments}.

\begin{theorem} \label{theorem:error-exact}
  Let $k > d/2$ and suppose that Assumption~\ref{assumption:regularity} holds and $c \leq 0$. If $\mathcal{H}(K, \Omega) \simeq H^k(\Omega)$, $m \in H^k(\Omega)$, and $f_t \in H^k(\R^d) \cap C^k(\R^d)$, then there is a positive constant $C_1$ that is independent of $X$ such that 
  \begin{equation} \label{eq:exact-posterior-bound-not-quasi-uniform}
    \mathbb{E} \big[ \norm[0]{ u_t - m_{u \mid \b{Y}} }_{L^2(\Omega)} \big] \leq C_1 \big( h_{X,\Omega}^{k+2} \sqrt{n} + h_{X, \Omega}^{d/2} \, n^{d/(4k)} \big) .
  \end{equation}
  If $h_{X,\Omega} = O(n^{-1/d}$), then there is a positive constant $C_2$ that is independent of $X$ such that
  \begin{equation} \label{eq:exact-posterior-bound-quasi-uniform}
    \mathbb{E} \big[ \norm[0]{ u_t - m_{u \mid \b{Y}} }_{L^2(\Omega)} \big] \leq C_2 \, n^{-1/2 + d/(4k)} 
  \end{equation}
  for all $n \geq 1$.
\end{theorem}
\begin{proof}[\normalfont\textbf{Proof sketch}] We first establish in \Cref{thm:Ku-RKHS} that the RKHS of $K_u$ is a sub-space of $H^{k+2}(\Omega)$ with an equivalent norm. The rest of the proof mirrors that of Corollary~3 in~\citep{Wynne2020} with minor modifications. The main ingredients are a sampling inequality from~\citep{Arcangeli2007} and bounds on the concentration function similar to those in~\citep{VaartZanten2011}.
\end{proof}

\begin{remark} \label{remark:optimal-rate}
  The mini-max optimal rate for regression in $H^k(\Omega)$ is  $n^{-k/(2k+d)}$~\citep[Chapter~2]{Tsybakov2009}.
  Since
  \begin{equation*}
    \frac{1}{2} - \frac{d}{4k} \leq \frac{1}{2} - \frac{d}{4k + 2d} = \frac{k}{2k+d},
  \end{equation*}
  the rate~\eqref{eq:exact-posterior-bound-quasi-uniform} that we have proved is slightly sub-optimal.
\end{remark}

Minor modifications to the proof of Theorem~\ref{theorem:error-exact-noiseless} yield a variant in which it is assumed that the observations are noiseless.
That is, $y_i = u_t(\b{x}_i)$ for $i \in \{ 1, \ldots, n \}$.
One can informally interpret this as the case $\sigma_\varepsilon = 0$.

\begin{theorem} \label{theorem:error-exact-noiseless}
  Consider the setting of Theorem~\ref{theorem:error-exact} but assume that the observations are noiseless.
  Then there is a positive constant $C_1$ that is independent of $X$ such that 
  \begin{equation*} 
    \mathbb{E} \big[ \norm[0]{ u_t - m_{u \mid \b{Y}} }_{L^2(\Omega)} \big] \leq C_1 h_{X,\Omega}^{k+2} \big( \norm[0]{ u_t }_{H^{k+2}(\Omega)} + \norm[0]{m_u}_{H^{k+2}(\Omega)} \big) .
  \end{equation*}
  If $h_{X,\Omega} = O(n^{-1/d}$), then there is a positive constant $C_2$ that is independent of $X$such that
  \begin{equation*} 
    \mathbb{E} \big[ \norm[0]{ u_t - m_{u \mid \b{Y}} }_{L^2(\Omega)} \big] \leq C_2 \, n^{-(k+2)/d} \big( \norm[0]{ u_t }_{H^{k+2}(\Omega)} + \norm[0]{m_u}_{H^{k+2}(\Omega)} \big)
  \end{equation*}
  for all $n \geq 1$.
\end{theorem}

It is straightforward to prove a similar noiseless variant of Theorem~\ref{thm:error-with-discrepancy}.
However, Theorem~\ref{theorem:error-exact-noiseless} does not extend to posteriors obtained via finite element discretisation because in the noiseless case it is difficult to bound the difference of posterior means corresponding to different kernels in terms of some distance between the said kernels.

\subsection{Finite element posterior} \label{eq:fe-error-estimates}

Next we turn to the analysis of the effect of finite element discretisation and consider the posterior mean in~\eqref{eq:fe-conditional-mean}.
A straightforward combination of Theorem~\ref{theorem:error-exact} and Proposition~\ref{prop:kernel-error} yields an error estimate that combines the errors from GP regression and finite element discretisation.

\begin{theorem} \label{theorem:error-fe}
  Let $k > d/2$.
  Suppose that Assumptions~\ref{assumption:regularity} and~\ref{assumption:fe-error} hold and that $c \leq 0$.
  If $\mathcal{H}(K, \Omega) \simeq H^k(\Omega)$, $m \in H^k(\Omega)$, $f_t \in H^k(\R^d) \cap C^k(\R^d)$, and $h_{X,\Omega} = O(n^{-1/d})$, then there is a positive constant $C$ that is independent of $X$ such that
  \begin{equation} \label{eq:error-gp-fe}
    \mathbb{E} \big[ \norm[0]{ u_t - m_{u \mid \b{Y}}^\FE }_{L^2(\Omega)} \big] \leq C \big( n^{-1/2 + d/(4k)} + ( n_\FE^{-q} + \sigma_\varepsilon^2) \sigma_\varepsilon^{-2} ( \norm[0]{f_t}_{L^\infty(\Omega)} + \sigma_\varepsilon ) n_\FE^{-q} n^{3/2} \big)
  \end{equation}
  for all $n \geq 1$.
\end{theorem}

\begin{remark} \label{rmk:n-relation}
  To obtain the best possible rate of convergence in terms of $n$ in~\eqref{eq:error-gp-fe}, we could set
  \begin{equation} \label{eq:best-n-fe}
    n_\FE = n^{(2 - d/(4k)) / q}.
  \end{equation}
  By incorporating all other terms in the constant $C$, we then obtain the error estimate
  \begin{equation*}
    \mathbb{E} \big[ \norm[0]{ u_t - m_{u \mid \b{Y}}^\FE }_{L^2(\Omega)} \big] \leq C \big( n^{-1/2 + d/(4k)} + n_\FE^{-q} n^{3/2} \big) = 2 C  n^{-1/2 + d/(4k)},
  \end{equation*}
  which is equal to the bound in~\eqref{eq:exact-posterior-bound-quasi-uniform} up to a constant factor.
\end{remark}

Practical application of Remark~\ref{rmk:n-relation} is difficult because, while~\eqref{eq:best-n-fe} yields the best possible polynomial rate in~\eqref{eq:error-gp-fe}, what one would actually like to obtain is the smallest possible right-hand side in~\eqref{eq:error-gp-fe}.
But finding $n_\FE$ that minimises the right-hand side is difficult because the constant factors involved are rarely, if ever, available.

\subsection{Inclusion of a discrepancy term}

Finally, we consider inclusion of a discrepancy term as described in Section~\ref{sec:discrepancy}.
The following two theorems concern the posterior means in~\eqref{eq:exact-conditional-mean-discrepancy} and~\eqref{eq:fe-conditional-mean-discrepancy}.
In these theorems it is assumed that the points are quasi-uniform, which means that there is $C_\textup{qu} > 0$ such that
\begin{equation*}
  q_X \leq h_{X, \Omega} \leq C_\textup{qu} q_X ,
\end{equation*}
where $h_{X, \Omega}$ is the fill-distance in~\eqref{eq:fill-distance} and $q_X$ is the separation radius
\begin{equation*}
  q_X = \frac{1}{2} \min_{i \neq j} \lVert \b{x}_i - \b{x}_j \rVert_2 .
\end{equation*}
Quasi-uniformity implies that the mesh ratio $\rho_{X, \Omega} = h_{X, \Omega} / q_X$ is uniformly bounded from above and that $h_{X,\Omega} = O(n^{-1/d})$; see Chapter~14 in~\citep{Wendland2005}.

\begin{theorem} \label{thm:error-with-discrepancy}
  Let $k_1 \geq r - 2 \geq k_2 > d/2$ and suppose that Assumption~\ref{assumption:regularity} holds with $k = k_1$ and $c \leq 0$.
  If $\mathcal{H}(K, \Omega) \simeq H^{k_1}(\Omega)$, $\mathcal{H}(K_d, \Omega) \simeq H^{r}(\Omega)$, $m_d \in H^{k_2+2}(\Omega)$, $m \in H^{k_2}(\Omega)$, and $f_t \in H^{k_2}(\R^d) \cap C^{k_2}(\R^d)$, then there is a positive constant $C_1$ that is independent of $X$ such that 
  \begin{equation} \label{eq:exact-posterior-disc-bound-not-quasi-uniform}
    \mathbb{E} \big[ \norm[0]{ u_t - m_{d;u \mid \b{Y}} }_{L^2(\Omega)} \big] \leq C_1 \big( h_{X,\Omega}^{k_2 + 2} \rho_{X,\Omega}^{r - k_2 - 2} + \sqrt{n} \, h_{X,\Omega}^{r} + n^{\kappa(k_2, r)} h_{X,\Omega}^{d/2} \big) .
  \end{equation}
  The constant $\kappa(k_2, r) \leq 1/2$ is given in~\eqref{eq:kappa-def}. If the points are quasi-uniform, there is a positive constant $C_2$ independent of $X$ such that
  \begin{equation} \label{eq:exact-posterior-disc-bound-quasi-uniform}
    \mathbb{E} \big[ \norm[0]{ u_t - m_{d;u \mid \b{Y}} }_{L^2(\Omega)} \big] \leq C_2 \big( n^{-(k_2+2)/d} + n^{-r/d + 1/2} + n^{-1/2 + \kappa(k_2, r)} \big)
  \end{equation}
  for all $n \geq 1$.
\end{theorem}
\begin{proof}[\normalfont\textbf{Proof sketch}] We proceed as in the proof of \Cref{theorem:error-exact}. However, now the assumption $k_1 \geq r - 2$ ensures that the $K_u$ is not rougher than $K_d$, which means that the regularity of the RKHS of $K_{ud} = K_u + K_d$ is determined by the regularity of $K_d$.
In this case we can directly apply Theorem~2 in~\citep{Wynne2020}.
\end{proof}

Combining \Cref{thm:error-with-discrepancy} with Proposition~\ref{prop:kernel-error} allows including the effect of finite element discretisation.

\begin{theorem} \label{thm:error-fe-with-discrepancy}
  Let $k_1 \geq r - 2 \geq k_2 > d/2$.
  Suppose that Assumptions~\ref{assumption:regularity} (with $k=k_1$) and~\ref{assumption:fe-error}  hold and that $c \leq 0$.
  If $\mathcal{H}(K, \Omega) \simeq H^{k_1}(\Omega)$, $\mathcal{H}(K_d, \Omega) \simeq H^{r}(\Omega)$, $m_d \in H^{k_2+2}(\Omega)$, $m \in H^{k_2}(\Omega)$, $f_t \in H^{k_2}(\R^d) \cap C^{k_2}(\R^d)$, and the points are quasi-uniform, then there is a positive constant $C$ independent of $X$ such that 
  \begin{equation} \label{eq:fe-posterior-disc-bound-quasi-uniform}
    \mathbb{E} \big[ \norm[0]{ u_t - m_{d;u \mid \b{Y}}^\FE }_{L^2(\Omega)} \big] \leq C \big( n^{-(k_2+2)/d} + n^{-r/d + 1/2} + n^{-1/2 + \kappa(k_2, r)} + n_\FE^{-q} n^{3/2}\big)
  \end{equation}
  for all $n \geq 1$.
  The constant $\kappa(k_2, r) \leq 1/2$ is given in~\eqref{eq:kappa-def}.
\end{theorem}

Unlike the results in Sections~\ref{eq:exact-error-estimates} and~\ref{eq:fe-error-estimates}, these theorems are valid also when the smoothness of the source term is misspecified.
That is, in Theorems~\ref{thm:error-with-discrepancy} and~\ref{thm:error-fe-with-discrepancy} it is possible that $k_1$, the smoothness of the kernel $K$ which specifies the prior for the source term, is larger than the smoothness, $k_2$, of the true source term~$f_t$.
Such misspecification results for GP regression in different settings can be found in, for example,~\citep{Karvonen2020, Kanagawa2020, Teckentrup2020, Wynne2020}.
The effectiveness of the discrepancy term in Theorems~\ref{thm:error-with-discrepancy} and~\ref{thm:error-fe-with-discrepancy} is due to the condition $k_1 \geq r - 2 \geq k_2$.
Because the prior on $u_t$ has regularity $k_1 + 2 \geq r$, the first inequality states that the discrepancy term is rougher than the $u_\GP$.
It is a general phenomenon in approximation and regression that the roughest component dominates, which is precisely what one can observe in the bounds~\eqref{eq:exact-posterior-disc-bound-not-quasi-uniform}--\eqref{eq:fe-posterior-disc-bound-quasi-uniform}.
The inequality $r - 2 \geq k_2$ states that $u_t$ is rougher than the discrepancy term, which is to say that the discrepancy term oversmooths the truth.
Oversmoothing is the scenario in which convergence rates tend to adapt to misspecification (see the references above).

\section{Numerical example} \label{sec:examples}

\begin{figure}[t]
  \centering
  \includegraphics[scale=0.9]{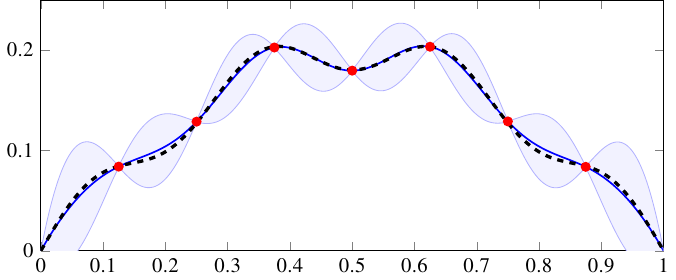}
  \caption{The statFEM conditional mean (blue) and 95\% credible interval given some noisy data (red points). The dashed black line is the true system response.} \label{fig:poisson-conditional-process}
\end{figure}

In this section we investigate the convergence of the posterior mean $m_{u \mid \b{Y}}^\FE$ to the true system response~$u_t$ for different values of the kernel smoothness parameter $k$.
We consider the one-dimensional Poisson's equation
\begin{equation} \label{eq:poisson-example}
  - u'' = f \quad \text{ in } \quad \Omega = (0, 1), \quad u(0) = u(1) = 0.
\end{equation}
The true source term is set as the function
\begin{equation} \label{eq:true-source-term-example}
  f_t(x) = \frac{\pi^2}{5} \sin(\pi x) + \frac{49\pi^2}{50} \sin(7 \pi x).
\end{equation}
The respective true system response is given in closed form by 
\begin{equation} \label{eq:true-system-response}
  u_t(x) = \frac{1}{5} \sin (\pi x) + \frac{1}{50} \sin(7 \pi x).
\end{equation}
A similar example was used in~\citep{Girolami2021}.

For the source term, we use a zero-mean Gaussian prior $f_\GP \sim \mathrm{GP}(0, K)$ with the Matérn covariance kernel~\eqref{eq:matern}.
We use the kernel hyperparameters
\begin{equation*}
  \ell \in \bigg\{ \frac{1}{2} , 1 \bigg\}, \quad \nu \in \bigg\{ \frac{1}{2}, \frac{5}{2} \bigg\}.
\end{equation*}
As explained in \Cref{sec:rkhs}, the values of $\nu$ correspond to the values $k \in \{1, 3\}$ of the RKHS smoothness parameter.
In order to facilitate comparison with standard GP regression based on a Matérn kernel, we set scaling parameter $\sigma$ such that the maximum of $K_u$ equals one (see Fig.\@~\ref{fig:intro-gp}).
For each $k$, the true source term in~\eqref{eq:true-source-term-example} is an element of the RKHS and of $C^k(\Omega)$. The selection of the hyperparameters could be automated by considering the marginal likelihood or cross-validation~\citep[Chapter 5]{williams2006gaussian}. 
For finite element analysis we use the standard piecewise linear basis functions (see Fig.\@~\ref{fig:finite-elements} for their bivariate versions) centered at $n_\FE \in \{32, 64, 128, 256\}$ uniformly placed points on $\Omega = (0, 1)$.
To compute the conditional mean $m_{u \mid \b{Y}}^\FE$, we use the approximation in~\eqref{eq:kernel-integral-approximation}; see Section~2.2 in~\citep{Girolami2021} for more details.
Observations of $u_t$ at \smash{$n \in \{2^1 - 1, 2^2 - 1, \ldots, n_\FE - 1\}$} uniformly placed points are corrupted by Gaussian noise with variances $\sigma_\varepsilon^2 \in \{ 10^{-2}, 10^{-4}\}$.
For illustration, the true system response and the finite element approximation to the conditional mean~\smash{$m_{u \mid \b{Y}}^\FE$} and the corresponding~$95\%$ credible interval are shown in Fig.\@~\ref{fig:poisson-conditional-process} for $\nu = 1/2$, $\ell = 1/2$, $n_\FE = 2048$, and $n = 7$.

\begin{figure}

  \centering
  \includegraphics[width=\textwidth]{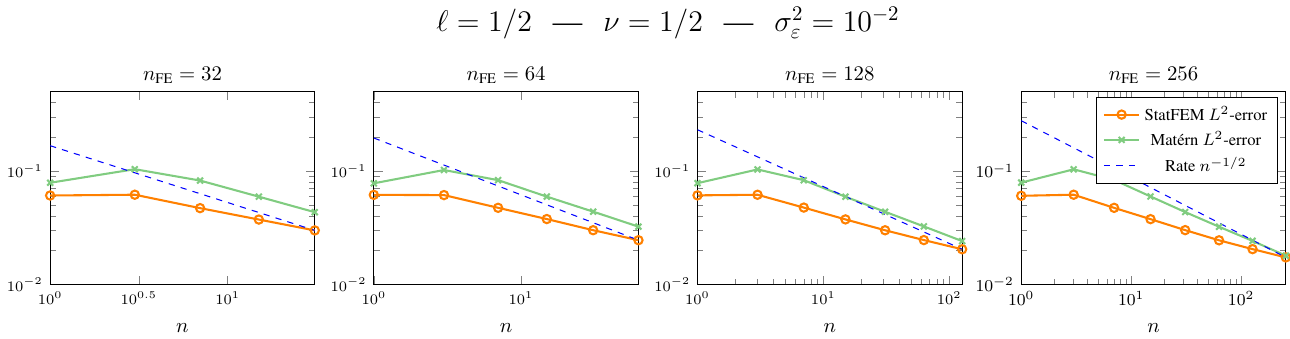}
  \includegraphics[width=\textwidth]{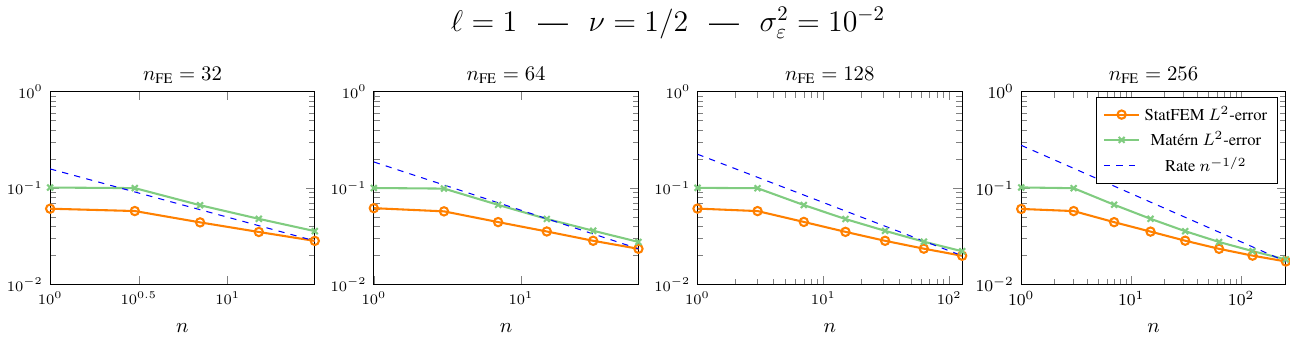}
  
  \caption{Convergence of the statFEM conditional mean (StatFEM $L^2$-error) \smash{$m_{u \mid \b{Y}}^\FE$} to $u_t$, true system response~\eqref{eq:true-system-response} of the Poisson's equation~\eqref{eq:poisson-example} with the source term~\eqref{eq:true-source-term-example}. 
  For comparison we also plot the convergence of the conditional mean when $u_t$ is modelled as a GP with a Matérn kernel (Matérn $L^2$-error).
  The plots show empirical approximations to expected $L^2$-errors over $10\,000$ noise realisations. The statFEM prior for the source term $f_t$ is a Matérn with parameters $\ell \in \{1/2, 1\}$ and $\nu = 1/2$. 
  The Matérn prior placed on $u_t$ has the same length-scales and smoothness $\nu + 2$, which ensures that both the induced statFEM prior and the Matérn prior for $u_t$ have the same regularity.
  Noise variance is $\sigma_\varepsilon^2 = 10^{-2}$.}
  \label{fig:conv1}
  
\end{figure}

\begin{figure}

  \centering
  \includegraphics[width=\textwidth]{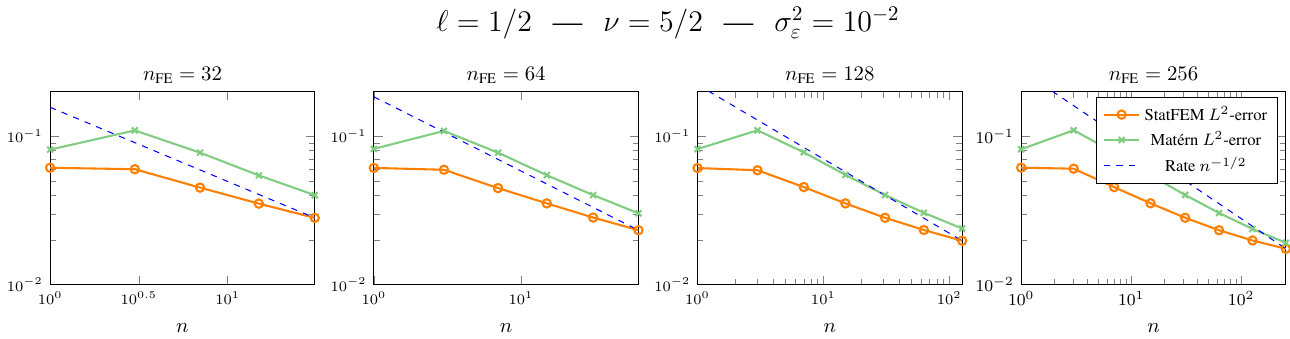}
  \includegraphics[width=\textwidth]{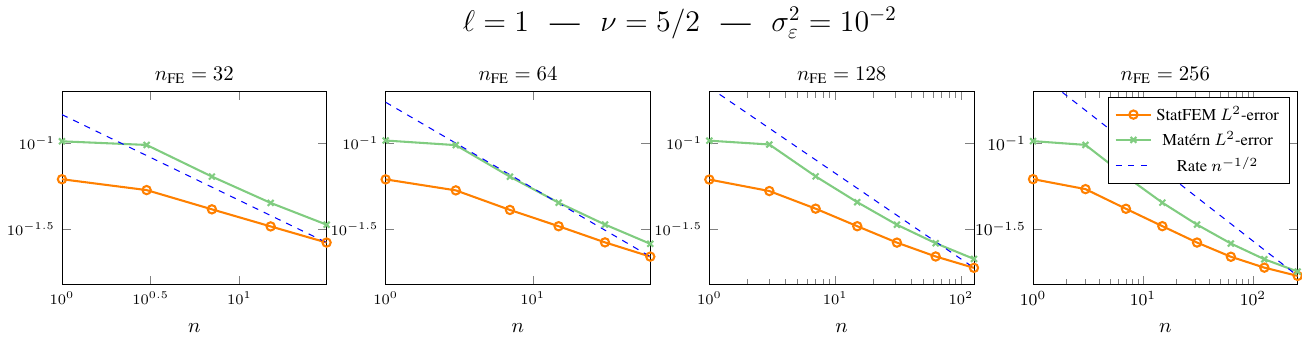}
  
  \caption{Convergence of the statFEM conditional mean (StatFEM $L^2$-error) \smash{$m_{u \mid \b{Y}}^\FE$} to $u_t$, true system response~\eqref{eq:true-system-response} of the Poisson's equation~\eqref{eq:poisson-example} with the source term~\eqref{eq:true-source-term-example}. 
  For comparison we also plot the convergence of the conditional mean when $u_t$ is modelled as a GP with a Matérn kernel (Matérn $L^2$-error).
  The plots show empirical approximations to expected $L^2$-errors over $10\,000$ noise realisations. The statFEM prior for the source term $f_t$ is a Matérn with parameters $\ell \in \{1/2, 1\}$ and $\nu = 5/2$. 
  The Matérn prior placed on $u_t$ has the same length-scales and smoothness $\nu + 2$, which ensures that both the induced statFEM prior and the Matérn prior for $u_t$ have the same regularity.
  Noise variance is $\sigma_\varepsilon^2 = 10^{-2}$.}
  \label{fig:conv2}
  
\end{figure}

Convergence results are depicted in Figs.\@~\ref{fig:conv1}--\ref{fig:conv4}.
In these results the $L^2$-norm is approximated by numerical quadrature and the expectation by an average over $10\,000$ independent observation noise realisations.
For each $\nu$ and $n_\FE$ we also plot the $L^2$-error of standard GP regression when $u_t$ is directly modelled as a purely data-driven GP whose kernel is a Matérn with smoothness $\nu + 2$ and parameters $\sigma_\textup{Matérn} = 1$ and $\ell_\textup{Matérn} = \ell$.
The selection of the smoothness parameter of the Matérn prior for $u_t$ corresponds to the smoothness of the induced prior $u_\GP$ in statFEM.
Being purely data-driven, this Matérn model for $u_t$ does not incorporate the boundary conditions or other structural characteristics.

We see that statFEM outperforms the Matérn model in Figs.\@~\ref{fig:conv1} to~\ref{fig:conv3}, particularly for small $n$.
This is to be expected as the prior dominates when there is little data.
In Fig.\@~\ref{fig:conv4} ($\ell \in \{1/2, 1\}$, $\nu=5/2$ and $\sigma_\varepsilon^2 = 10^{-4}$), statFEM exhibits clear saturation.
However, as the Matérn model behaves similarly when $\ell = 1$, $\nu = 5/2$ and $\sigma_\varepsilon^2 = 10^{-4}$, it seems that the saturation effect is not specific to statFEM in this example.
The plots also show that statFEM works well even when the number of finite element nodes, $n_\FE$, is small.
This is important because, especially in higher dimensions, the number of data points will be significantly smaller than the number of finite element nodes so that it will become even more important to encode the PDE and its boundary conditions. 

\begin{figure}

  \centering
  \includegraphics[width=\textwidth]{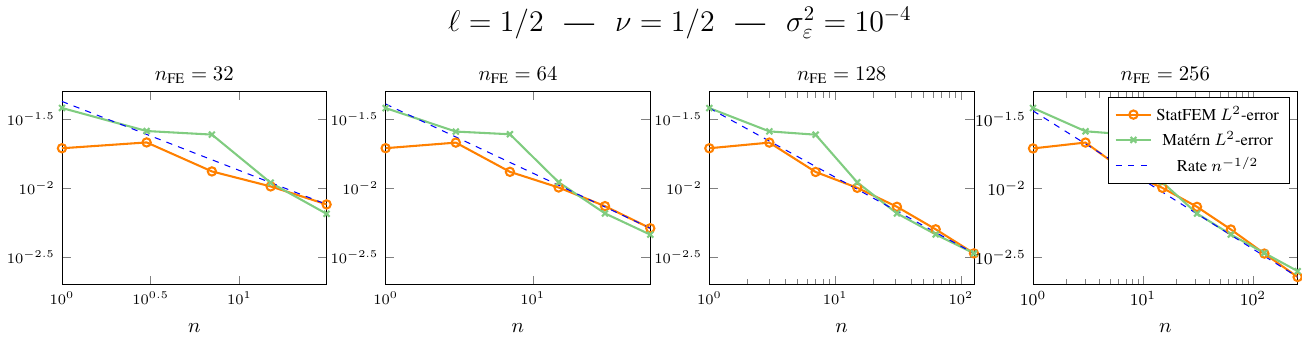}
  \includegraphics[width=\textwidth]{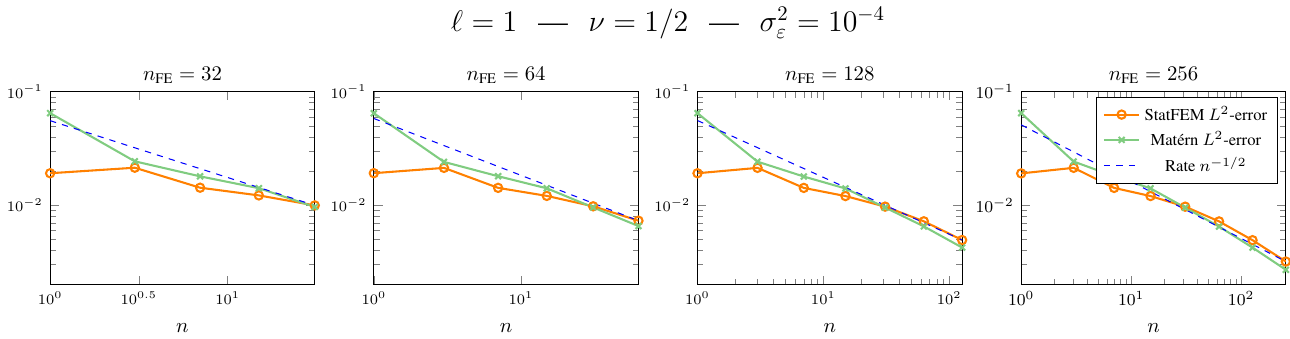}
  
  \caption{Convergence of the statFEM conditional mean (StatFEM $L^2$-error) \smash{$m_{u \mid \b{Y}}^\FE$} to $u_t$, true system response~\eqref{eq:true-system-response} of the Poisson's equation~\eqref{eq:poisson-example} with the source term~\eqref{eq:true-source-term-example}. 
  For comparison we also plot the convergence of the conditional mean when $u_t$ is modelled as a GP with a Matérn kernel (Matérn $L^2$-error).
  The plots show empirical approximations to expected $L^2$-errors over $10\,000$ noise realisations. The statFEM prior for the source term $f_t$ is a Matérn with parameters $\ell \in \{1/2, 1\}$ and $\nu = 1/2$. 
  The Matérn prior placed on $u_t$ has the same length-scales and smoothness $\nu + 2$, which ensures that both the induced statFEM prior and the Matérn prior for $u_t$ have the same regularity.
  Noise variance is $\sigma_\varepsilon^2 = 10^{-4}$.}
  \label{fig:conv3}
  
\end{figure}

\begin{figure}

  \centering
  \includegraphics[width=\textwidth]{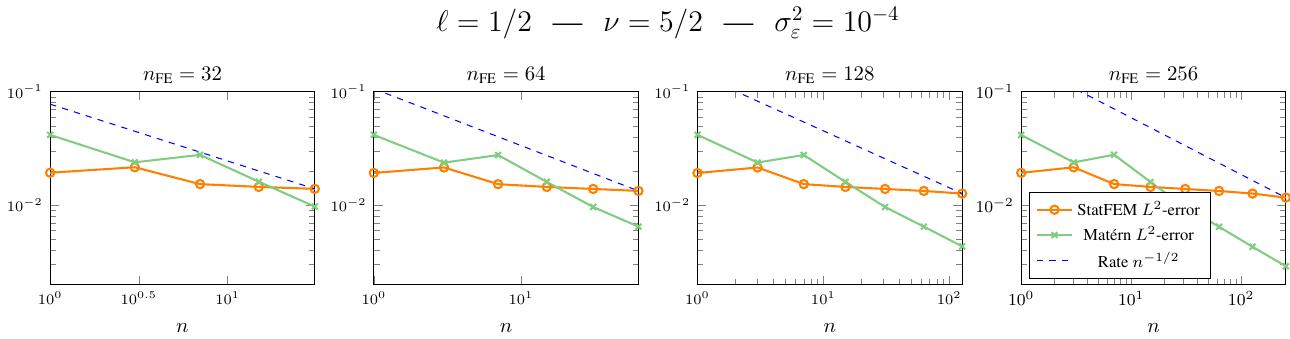}
  \includegraphics[width=\textwidth]{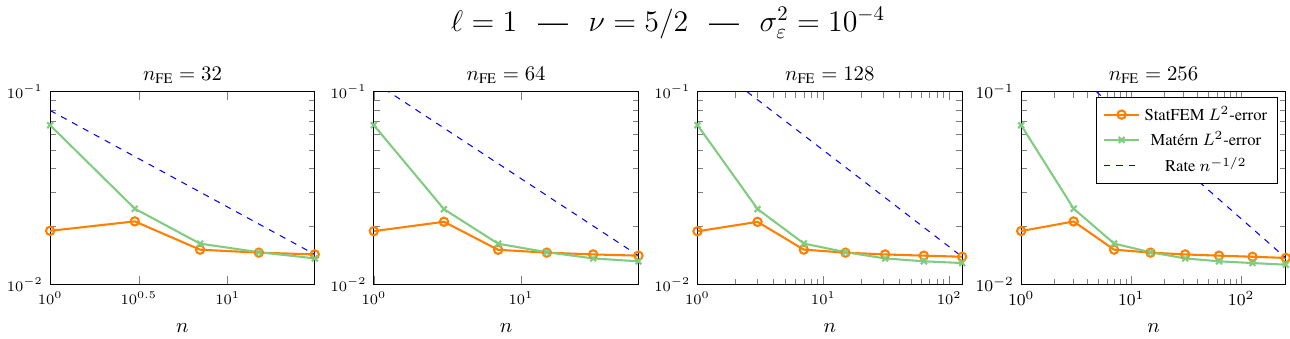}
  
  \caption{Convergence of the statFEM conditional mean (StatFEM $L^2$-error) \smash{$m_{u \mid \b{Y}}^\FE$} to $u_t$, true system response~\eqref{eq:true-system-response} of the Poisson's equation~\eqref{eq:poisson-example} with the source term~\eqref{eq:true-source-term-example}. 
  For comparison we also plot the convergence of the conditional mean when $u_t$ is modelled as a GP with a Matérn kernel (Matérn $L^2$-error).
  The plots show empirical approximations to expected $L^2$-errors over $10\,000$ noise realisations. The statFEM prior for the source term $f_t$ is a Matérn with parameters $\ell \in \{1/2, 1\}$ and $\nu = 5/2$. 
  The Matérn prior placed on $u_t$ has the same length-scales and smoothness $\nu + 2$, which ensures that both the induced statFEM prior and the Matérn prior for $u_t$ have the same regularity.
  Noise variance is $\sigma_\varepsilon^2 = 10^{-4}$.}
  \label{fig:conv4}
  
\end{figure}

\section{Concluding Remarks}

We have analysed a particular formulation of the statFEM approach in a deterministic setting with generic points and finite elements.
A different set of assumptions could be used---the practical relevance of these assumptions would likely depend much on the application and whether or not the user views the data-generating process as an actual Gaussian process or some unknown deterministic function.
Settings that we believe could or could not be analysed using similar or related techniques as those in this article include the following:
\begin{itemize}
\item We have considered a ``mixed'' case in which a GP is used to model a deterministic function. But one could alternatively assume that $f_t$, and consequently $u_t$, is a GP (or some other stochastic process) and proceed from there. This is how statFEM is  formulated in~\cite{Girolami2021}.
\item Distribution of the points $\b{x}_1, \ldots, \b{x}_n$ where the measurement data are obtained is quite generic in this article in that no reference is made to how these points might be selected or sampled and all results are formulated in terms of the fill-distance $h_{X,\Omega}$ or, in the quasi-uniform case, $n$. In applications the points may be sampled randomly from some distribution on $\Omega$. We refer to~\citep[Theorem~1]{Briol2019} and~\citep{KriegSonnleitner2024} for related results concerning random points.
\item To remove the assumption that the measurement data are noisy is likely to be challenging if one is interested in including the effect of the finite element discretisation. It is straightforward to derive versions of Theorems~\ref{theorem:error-exact} and~\ref{thm:error-with-discrepancy} in the noiseless case (see Theorem~\ref{theorem:error-exact-noiseless}), but no other result in Section~\ref{sec:error-estimates} generalises readily. The reason is the presence of the factor $\sigma^{-2}$ in~\eqref{eq:kernel-error} which is used to prove all theorems concerning finite element discretisations: if $\sigma = 0$, this bound is rendered meaningless.
\item As already mentioned in Section~\ref{eq:fe-error-estimates}, the bounds  include a variety of non-explicit constants. We do not believe that the constants are computable in all but perhaps the simplest of special cases.
\item Although desirable for computational reasons, the Gaussian process framework that we have used is restrictive. It may be possible to derive rates under more general assumptions~\citep{Ghosal2000}. However, the proof technique would have to be different as our proof relies on sampling inequalities and a connection between the GP posterior and approximation in an RKHS.
\item An interesting problem that we have not encountered in the GP or kernel literature would be to recover $f_t$ from observations of $u_t$. Because $f_t = \mathcal{L}^{-1} u_t$, this is a generalised regression problem in which one tries to recover a function from linear information.
Alternatively, this is a linear inverse problem in which one recovers the source from the response.
If a GP prior with covariance $K$ is placed on $f_t$, conditioning on evaluations of $u_t$ at $X$ gives rise to a posterior that involves the kernel matrix $\b{K}_u(X, X)$; see~\citep{Travelletti2024}. Some of the results derived in this article could therefore be useful.
\end{itemize}

\section{Proofs} \label{sec:proofs}

This section contains the proofs of the theorems in Section~\ref{sec:results}.
We first collect and derive a number of auxiliary results that we use to prove the error estimates.
These results are of four types: (i) standard regularity results for solutions of elliptic PDEs; (ii) results on the RKHS of the kernel $K_u$; (iii) sampling inequalities; and (iv) results related to the concentration function and small ball probabilities of Gaussian measures.
We use the notation $C = C(\theta_1, \ldots, \theta_p)$ to indicate that a constant $C$ depends only on some parameters $\theta_1, \ldots, \theta_p$.

\subsection{Regularity results for elliptic pdes} \label{sec:pde-regularity}

Recall the function spaces from Section~\ref{sec:rkhs}.
In Assumption~\ref{assumption:regularity} and Theorem~\ref{thm:pde-uniform-bound}, the boundary of $\Omega$ is required to have certain smoothness.
The boundary $\partial \Omega$ is said to be $C^k$ (or $C^{k,\alpha})$ if for each $\b{x}_0 \in \partial \Omega$ there is an open ball $B$ centered at $\b{x}_0$ and a surjection $\psi$ from $B$ to some $D \subset \R^d$ such that (a) $\psi(B \cap \Omega) \subset \R_+^d$, (b) $\psi(B \cap \partial \Omega) \subset \partial \R_+^d$, (c) $\psi \in C^k(\Omega)$ [resp.\@~$\psi \in C^{k,\alpha}(\Omega)$], and (d) $\psi^{-1} \in C^k(D)$ [resp.\@~$\psi \in C^{k,\alpha}(D)$].
Certain standard regularity results and estimates play a crucial role in the derivation of our results.
The following regularity theorem can be found in, for example,~\citep[Theorem~5 in Section~6.3]{Evans1998}.

\begin{theorem} \label{thm:pde-regularity}
  Consider the elliptic PDE given in~\eqref{eq:elliptic-pde}.  
  Let $k \in \N_0$ and suppose that Assumption~\ref{assumption:regularity} holds.
  If $f \in H^k(\Omega) \cap \mathcal{H}_\mathcal{L}(\Omega)$, then $u \in H^{k+2}(\Omega)$ and there is a constant $C = C(k, \Omega, \mathcal{L})$ such that
  \begin{equation*}
    \norm[0]{u}_{H^{k+2}(\Omega)} \leq C \norm[0]{f}_{H^k(\Omega)}.
  \end{equation*}
\end{theorem}

The following boundedness result is a combination of the a priori bound in Theorem~3.7 and the Schauder regularity result in Theorem~6.14 of~\citep{GilbargTrudinger1983}.

\begin{theorem} \label{thm:pde-uniform-bound}
  Consider the elliptic PDE in~\eqref{eq:elliptic-pde}.
  Suppose that $\partial\Omega$ is $C^{2,\alpha}$, $a_{ij}, b_i, c \in C^{0,\alpha}(\bar{\Omega})$ for all $i,j \in \{ 1,\ldots,d \}$ and some $\alpha \in (0, 1)$ and $c \leq 0$.
  If $f \in C^{0, \alpha}(\bar{\Omega})$, then $u \in C^{2,\alpha}(\bar{\Omega})$ and there is a constant $C = C(\Omega, \mathcal{L})$ such that
  \begin{equation*}
    \norm[0]{ u }_{L^\infty(\Omega)} \leq C \norm[0]{ f }_{L^\infty(\Omega)}.
  \end{equation*}
\end{theorem}

Note that Assumption~\ref{assumption:regularity} implies the regularity assumptions in Theorem~\ref{thm:pde-uniform-bound}.

\subsection{Transformed reproducing kernel Hilbert spaces} \label{sec:rkhs-trans}

The following lemma justifies the assumption in~\eqref{eq:kernel-HL-assumption} that $\mathcal{L}_{\b{x}}^{-1} K(\cdot, \b{x})$ is an element of $\mathcal{H}(K, \Omega) \subset \mathcal{H}_\mathcal{L}(\Omega)$ for every $\b{x} \in \Omega$.
Let $\delta_{\b{x}}$ be the point evaluation functional at $\b{x} \in \Omega$, which is to say that $\delta_{\b{x}}(f) = f(\b{x})$.
As earlier, whenever there is a risk of ambiguity we use subscripts to denote the variable to which a functional or an operator applies to.
That is, 
\begin{equation*}
  \mathcal{A} f(\b{x}) = \mathcal{A}_{\b{x}} f(\b{x}) = (\delta_{\b{x}} \circ \mathcal{A})_{\b{x}'} f(\b{x}')
\end{equation*}
for any operator $\mathcal{A}$.

\begin{lemma} \label{lemma:bounded-functionals}
  Let $k > d/2$.
  Suppose that $\mathcal{H}(K, \Omega) \simeq H^k(\Omega)$ and that Assumption~\ref{assumption:regularity} holds.
  Then the functional $\delta_{\b{x}} \circ \mathcal{L}^{-1}$ is bounded on $\mathcal{H}(K, \Omega)$ and $\mathcal{L}_{\b{x}}^{-1} K(\cdot, \b{x}) \in \mathcal{H}(K, \Omega)$ for every $\b{x} \in \Omega$.
\end{lemma}
\begin{proof}
  Let $\b{x} \in \Omega$ and set $\ell_{\b{x}} = \delta_{\b{x}} \circ \mathcal{L}^{-1}$.
  Since $\mathcal{H}(K, \Omega)$ is norm-equivalent to $H^k(\Omega)$ and $H^{k+2}(\Omega)$ is continuously embedded in $C(\Omega)$, it follows from Theorem~\ref{thm:pde-regularity} that
  \begin{equation*}
    \abs[0]{ \ell_{\b{x}}(f) } = \abs[0]{ u(\b{x}) } \leq \norm[0]{u}_{L^\infty(\Omega)} \leq C_1 \norm[0]{u}_{H^{k+2}(\Omega)} \leq C_1 C_2 \norm[0]{f}_{H^k(\Omega)} \leq C_1 C_2 C_K^{-1} \norm[0]{f}_K
  \end{equation*}
  for any $f \in \mathcal{H}(K, \Omega)$ and constants $C_1 = C(k, \Omega)$ and $C_2 = C(k, \Omega, \mathcal{L})$.
  This proves that $\ell_{\b{x}}$ is bounded on $\mathcal{H}(K, \Omega)$.
  Because $\ell_{\b{x}}$ is bounded, it follows from the Riesz representation theorem that there exists a unique function $l_{\b{x}} \in \mathcal{H}(K, \Omega)$ such that $\ell_{\b{x}}(f) = \inprod{f}{l_{\b{x}}}_K$ for every $f \in \mathcal{H}(K, \Omega)$. Setting $f = K(\cdot, \b{y})$ and using the reproducing property~\eqref{eq:reproducing-property} we get, for any $\b{y} \in \Omega$,
  \begin{equation*}
    \mathcal{L}_{\b{x}}^{-1} K(\b{y}, \b{x}) = \ell_{\b{x}}(K(\b{y}, \b{x})) = \inprod{K(\b{y}, \cdot)}{l_{\b{x}}}_K = l_{\b{x}}(\b{y}).
  \end{equation*}
  That is, $\mathcal{L}_{\b{x}}^{-1} K(\cdot, \b{x}) = l_{\b{x}} \in \mathcal{H}(K, \Omega)$.
\end{proof}

Next we want to understand how the RKHS of the kernel $K_u$ defined in~\eqref{eq:Ku-definition} relates to that of $K$.
We use the following general proposition about transformations of RKHSs.
See Theorems~16.8 and~16.9 in~\citep{Wendland2005}, Section~5.4 in~\citep{Paulsen2016}, and Section~7 in~\citep{VaartZanten2008} for similar results.
A proof is included here for completeness and because our formulation differs slightly from those that we have found in the literature.

\begin{proposition} \label{prop:rkhs-transformation}
  Let $K$ be a positive-semidefinite kernel on $\Omega$ and $\mathcal{A}$ an invertible linear operator on $\mathcal{H}(K, \Omega)$ such that the functional $\delta_{\b{x}} \circ \mathcal{A}$ is bounded on $\mathcal{H}(K, \Omega)$ for every $\b{x} \in \Omega$.
  Then
  \begin{equation*}
    R( \b{x}, \b{y} ) = \mathcal{A}_{\b{x}} \mathcal{A}_{\b{y}} K(\b{x}, \b{y}) = (\delta_{\b{x}} \circ \mathcal{A})_{\b{x}'} (\delta_{\b{y}} \circ \mathcal{A})_{\b{y}'} K(\b{x}', \b{y}')
  \end{equation*}
  defines a positive-semidefinite kernel on $\Omega$.
  Furthermore,
  \begin{equation*}
    \mathcal{H}(R, \Omega) = \mathcal{A} (\mathcal{H}(K, \Omega)) = \Set{ \mathcal{A}f }{ f \in \mathcal{H}(K, \Omega)} \quad \text{ and } \quad \norm[0]{\mathcal{A} f}_R = \norm[0]{f}_K \: \text{ for every } \: f \in \mathcal{H}(K, \Omega).
  \end{equation*}
\end{proposition}
\begin{proof}
  Because the functional $\ell_{\b{y}} = \delta_{\b{y}} \circ \mathcal{A}$ is bounded on $\mathcal{H}(K, \Omega)$, the Riesz representation theorem implies that there exists a unique representer $l_{\b{y}} \in \mathcal{H}(K, \Omega)$ such that $\ell_{\b{y}}(f) = \inprod{f}{l_{\b{y}}}_K$ for every $f \in \mathcal{H}(K, \Omega)$.
  Therefore, by the reproducing property,
  \begin{equation*}
    l_{\b{y}}(\b{x}) = \inprod{K(\b{x}, \cdot)}{l_{\b{y}}}_K = \ell_{\b{y}}(K(\b{x}, \cdot)) = (\delta_{\b{y}} \circ \mathcal{A})_{\b{u}} K(\b{x}, \b{u})
  \end{equation*}
  for any $\b{x}, \b{y} \in \Omega$.
  Since $l_{\b{y}}$ is an element of $\mathcal{H}(K, \Omega)$, $\ell_{\b{x}}(l_{\b{y}}) = \inprod{l_{\b{x}}}{l_{\b{y}}}_K$ and
  \begin{equation*}
    \ell_{\b{x}}(l_{\b{y}}) = (\delta_{\b{x}} \circ \mathcal{A})(l_{\b{y}}) = (\delta_{\b{x}} \circ \mathcal{A})_{\b{v}} (\delta_{\b{y}} \circ \mathcal{A})_{\b{u}} K(\b{v}, \b{u}) = R(\b{x}, \b{y}),
  \end{equation*}
  from which it follows that $R$ is a well-defined kernel.
  To verify that $R$ is positive-semidefinite, compute
  \begin{equation} \label{eq:transformed-kernel-psd}
      \sum_{i=1}^N \sum_{j=1}^N a_i a_j R(\b{x}_i, \b{x}_j) = \sum_{i=1}^N \sum_{j=1}^N a_i a_j \inprod{l_{\b{x}_i}}{l_{\b{x}_j}}_K = \inprod[\Bigg]{ \sum_{i=1}^N a_i l_{\b{x}_i} }{ \sum_{i=1}^N a_i l_{\b{x}_i} }_K = \norm[4]{\sum_{i=1}^N a_i l_{\b{x}_i} }_K^2 \geq 0
  \end{equation}
  for any $N \in \N$, $a_i \in \R$, and $\b{x}_i \in \Omega$.
  To prove the claims related to $\mathcal{H}(R, \Omega)$ we use a classical characterisation~\citep[e.g.,][Section~3.4]{Paulsen2016} which states that $f \in \mathcal{H}(K, \Omega)$ if and only if there is $c > 0$ such that
  \begin{equation} \label{eq:Kc-kernel}
    K_c(\b{x}, \b{y}) = c^2 K(\b{x}, \b{y}) - f(\b{x}) f(\b{y})
  \end{equation}
  defines a positive-semidefinite kernel. The smallest constant for which $K_c$ is positive-semidefinite equals the RKHS norm of $f$.
  Now, assuming that $f \in \mathcal{H}(K, \Omega)$ and applying $\mathcal{A}$ twice on~\eqref{eq:Kc-kernel} yields the kernel
  \begin{equation*}
    R_c(\b{x}, \b{y}) = c^2 R(\b{x}, \b{y}) - \mathcal{A}f(\b{x}) \mathcal{A}f(\b{y}),
  \end{equation*}
  which, by the argument used in~\eqref{eq:transformed-kernel-psd}, is positive-semidefinite if $K_c$ is.
  Hence $\mathcal{A}(\mathcal{H}(K, \Omega)) \subset \mathcal{H}(R, \Omega)$ and $\norm[0]{\mathcal{A} f}_R \leq \norm[0]{f}_K$.
  That these are equalities follows from the invertibility of $\mathcal{A} \colon \mathcal{H}(K, \Omega) \to \mathcal{A}(\mathcal{H}(K, \Omega))$.
\end{proof}

Applying Proposition~\ref{prop:rkhs-transformation} to $\mathcal{A} = \mathcal{L}^{-1}$ yields the following theorem. 

\begin{theorem} \label{thm:Ku-RKHS}
  Let $k > d/2$ and consider the kernel $K_u$ in~\eqref{eq:Ku-definition}.
  If $\mathcal{H}(K, \Omega) \simeq H^k(\Omega)$ and Assumption~\ref{assumption:regularity} holds, then
  \begin{itemize}
  \item[\normalfont (i)] The kernel $K_u$ is positive-semidefinite on $\Omega$ and its RKHS is
    \begin{equation} \label{eq:Ku-RKHS}
      \mathcal{H}(K_u, \Omega) = \Set{ u }{ \text{$u$ is a solution of~\eqref{eq:elliptic-pde} for some $f \in \mathcal{H}(K, \Omega)$}}.
      \end{equation}
      Furthermore, $\norm[0]{u}_{K_u} = \norm[0]{f}_K$.
  \item[\normalfont (ii)] It holds that $\mathcal{H}(K_u, \Omega) \subset H^{k+2}(\Omega)$ and there are
    constants $C_u = C(K, k, \Omega, \mathcal{L})$ and $C_u' = C(K, k, \Omega, \mathcal{L})$ such that
  \begin{equation} \label{eq:Ku-norm-equivalence}
    C_u \norm[0]{u}_{H^{k+2}(\Omega)} \leq \norm[0]{u}_{K_u} \leq C_u' \norm[0]{u}_{H^{k+2}(\Omega)}
  \end{equation}
  for all $u \in \mathcal{H}(K_u, \Omega)$.
  \end{itemize}
\end{theorem}
\begin{proof}
  Because $\mathcal{H}(K, \Omega) \subset \mathcal{H}_\mathcal{L}(\Omega)$, the linear operator $\mathcal{L}$ is invertible on $\mathcal{H}(K, \Omega)$.
  Furthermore, by Lemma~\ref{lemma:bounded-functionals} the functionals $\delta_{\b{x}} \circ \mathcal{L}^{-1}$ are bounded on $\mathcal{H}(K, \Omega)$.
  Therefore the first claim follows by applying Proposition~\ref{prop:rkhs-transformation} to $\mathcal{A} = \mathcal{L}^{-1}$.
  To verify the second claim, observe that it now follows from the norm-equivalence $\mathcal{H}(K, \Omega) \simeq H^k(\Omega)$ and Theorem~\ref{thm:pde-regularity} that, for a constant $C = C(k, \Omega, \mathcal{L})$,
  \begin{equation*}
    \norm[0]{u}_{K_u} = \norm[0]{f}_{K} \geq C_K \norm[0]{f}_{H^k(\Omega)} \geq C_K C^{-1} \norm[0]{u}_{H^{k+2}(\Omega)}
  \end{equation*}
  and
  \begin{equation*}
    \norm[0]{u}_{K_u} = \norm[0]{f}_{K} \leq C_K' \norm[0]{f}_{H^k(\Omega)} = C_K' \norm[0]{\mathcal{L} u}_{H^k(\Omega)} \leq C_K' C_\mathcal{L} \norm[0]{u}_{H^{k+2}(\Omega)},
  \end{equation*}
   where $C_\mathcal{L} = C_\mathcal{L}(k, \Omega, \mathcal{L})$ and the last inequality follows from the fact that the differential operator $\mathcal{L}$ is second-order and its coefficient functions are in $C^{k+1}(\bar{\Omega})$ by Assumption~\ref{assumption:regularity}.
\end{proof}

\Cref{thm:Ku-RKHS} is essential in what follows. 
We believe that there is room for significant generalisations by studying operators $\mathcal{A}$ for which the transformed RKHS can be connected to a Sobolev (or other classical) function space.
Finally, we need the following result~\citep[e.g.,][p.\@~24]{Berlinet2004} on the RKHS of a sum kernel to analyse statFEM when a discrepancy term is included (recall Section~\ref{sec:discrepancy}).

\begin{theorem} \label{thm:rkhs-sum}
  Let $K_1$ and $K_2$ be two positive-semidefinite kernels on $\Omega$.
  Then $R = K_1 + K_2$ is a positive-semidefinite kernel on $\Omega$ and its RKHS consists of functions which can be written as $f = f_1 + f_2$ for $f_1 \in \mathcal{H}(K_1, \Omega)$ and $f_2 \in \mathcal{H}(K_2, \Omega)$.
  The RKHS norm is
  \begin{equation*}
    \norm[0]{f}_R^2 = \min \Set[\big]{ \norm[0]{f_1}_{K_1}^2 + \norm[0]{f_2}_{K_2}^2 }{ f = f_1 + f_2 \: \text{ s.t. } \: f_1 \in \mathcal{H}(K_1, \Omega), \: f_2 \in \mathcal{H}(K_2, \Omega) } .
  \end{equation*}
\end{theorem}

\subsection{Sampling inequalities}

Denote $(x)_+ = \max\{x,0\}$ for $x \in \R$.
The following sampling inequality is the main building block of our error estimates.

\begin{theorem}[Theorem~4.1 in \citep{Arcangeli2007}] \label{thm:arcangeli}
  Let $p \in [1, \infty]$, $k > d/2$, and $\gamma = \max\{p, 2\}$.
  If $g \in H^k(\Omega)$, then there is a constant $C = C(k, p, \Omega)$ such that
  \begin{equation} \label{eq:arcangeli-bound}
    \norm[0]{g}_{L^p(\Omega)} \leq C \Big( h_{X,\Omega}^{k-d(1/2-1/p)_+} \norm[0]{g}_{H^{k}(\Omega)} + h_{X,\Omega}^{d/\gamma} \norm[0]{\b{g}(X)}_2 \Big) .
  \end{equation}
  Here $\b{g}(X) = (g(\b{x}_1), \ldots, g(\b{x}_n)) \in \R^n$.
\end{theorem}

See \citep{WangTuoWu2020, Karvonen2020, Teckentrup2020, Wynne2020, Wang2021} for a variety of applications of this and related sampling inequalities to error analysis of GP regression.
Note that it is often stated that~\eqref{eq:arcangeli-bound} and similar bounds hold when $h_{X,\Omega}$ is sufficiently small, a requirement that can be eliminated by enlarging the multiplicative constant $C$ on the right-hand side.

\subsection{The concentration function and small ball probabilities}

Final ingredients that we need are certain results on the concentration function and small ball probabilities of Gaussian measures.
Given $u_0 = \mathcal{L}^{-1} f_0$ for some $f_0 \in \mathcal{H}_\mathcal{L}(\Omega)$, define the concentration function
\begin{equation*}
  \phi_{u_0}(\varepsilon) = \gamma_{u_{u_0}}(\varepsilon) + \beta(\varepsilon),
\end{equation*}
where
\begin{equation*}
  \gamma_{u_0}(\varepsilon) = \inf_{ u \in \mathcal{H}(K_u, \Omega)} \Set[\big]{ \norm[0]{u}_{K_u}^2 }{ \norm[0]{u - u_0}_{L^\infty(\Omega)} < \varepsilon}, \quad\quad \beta(\varepsilon) = - \ln \Pi_u \big( \Set{ u }{ \norm[0]{u}_{L^\infty(\Omega)} < \varepsilon} \big).
\end{equation*}
Here $\Pi_u$ denotes the Gaussian probability measure associated to the zero-mean Gaussian process with covariance kernel $K_u$.

\begin{proposition} \label{prop:decentering-term}
  Let $k \geq r > d/2$.
  Suppose that $\mathcal{H}(K, \Omega) \simeq H^k(\Omega)$, that Assumption~\ref{assumption:regularity} holds and that $c \leq 0$.
  If there exists $f_0 \in H^r(\R^d) \cap C^r(\R^d)$ such that $u_0 = \mathcal{L}^{-1} f_0 |_\Omega$, then there is a constant $C = C(f_0, K, k, r, \Omega, \mathcal{L})$ such that
  \begin{equation*}
    \gamma_{u_0}(\varepsilon) \leq C \varepsilon^{-2(k-r)/r}
  \end{equation*}
  when $\varepsilon$ is sufficiently small.
\end{proposition}
\begin{proof}
  Theorem~\ref{thm:Ku-RKHS} implies that 
  \begin{equation*}
    \begin{split}
      \gamma_{u_0}(\varepsilon) &= \inf_{ f \in \mathcal{H}(K, \Omega)} \Set[\big]{ \norm[0]{f}_{K} }{ \norm[0]{u - u_0}_{L^\infty(\Omega)} < \varepsilon \: \text{ for } \: u = \mathcal{L}^{-1} f}.
      \end{split}
  \end{equation*}
  Since $k > d/2$, so that $\mathcal{H}(K, \Omega)$ is embedded in a Hölder space, and $r \geq 1$, for any $f \in \mathcal{H}(K, \Omega)$ the function $f - f_0|_\Omega$ has a unique continuous extension to $\bar{\Omega}$ that satisfies the assumptions of Theorem~\ref{thm:pde-uniform-bound}. Thus there is $C_1 = C(\Omega, \mathcal{L})$ such that \smash{$\norm[0]{ u - u_0 }_{L^\infty(\Omega)} \leq C_1 \norm[0]{ f - f_0|_\Omega}_{L^\infty(\Omega)}$}.
  Therefore \smash{$\norm[0]{ u - u_0 }_{L^\infty(\Omega)} < \varepsilon$} if \smash{$\norm[0]{ f - f_0|_\Omega}_{L^\infty(\Omega)} < \varepsilon C_1^{-1}$}, which implies that
  \begin{equation*}
    \gamma_{u_0}(\varepsilon) \leq \inf_{ f \in \mathcal{H}(K, \Omega)} \Set[\big]{ \norm[0]{f}_{K} }{ \norm[0]{f - f_0|_\Omega}_{L^\infty(\Omega)} < \varepsilon/C_1 }.
  \end{equation*}
  Lemma~4 in~\citep{VaartZanten2011} and Lemma~23 in~\citep{Wynne2020} bound the right-hand side as
  \begin{equation*}
    \inf_{ f \in \mathcal{H}(K, \Omega)} \Set[\big]{ \norm[0]{f}_{K} }{ \norm[0]{f - f_0|_\Omega}_{L^\infty(\Omega)} < \varepsilon/C_1 } \leq C_2 C_1^{2(k-r)/r} \varepsilon^{-2(k-r)/r}
  \end{equation*}
  for $C_2 = C(f_0, K, k, r, \Omega)$ when $\varepsilon$ is sufficiently small.
  This completes the proof.
\end{proof}

Let $\varepsilon > 0$ and let $( \mathcal{H}, d_\mathcal{H})$ be a metric space
The metric entropy of a compact subset $A$ of $( \mathcal{H}, d_\mathcal{H})$ is defined as $H_\textup{ent} (A, d_\mathcal{H}, \varepsilon) = \ln N(A, d_\mathcal{H}, \varepsilon)$.
Here $N(A, d_\mathcal{H}, \varepsilon)$ denotes the minimum covering number:
\begin{equation*}
  N(A, d_\mathcal{H}, \varepsilon) = \min \Set[\Bigg]{ n \geq 1 }{ \text{there exist } \: x_1, \ldots, x_n \in A \: \text{ such that } \: A \subset \bigcup_{i=1}^n B_\varepsilon(x_i; \mathcal{H}, d_\mathcal{H})},
\end{equation*}
where $B_\varepsilon(x; \mathcal{H}, d_\mathcal{H})$ is the $x$-centered $\varepsilon$-ball in $(\mathcal{H}, d_\mathcal{H})$.
If $( \mathcal{H}, d_\mathcal{H})$ is a normed space, we have the scaling identity
\begin{equation} \label{eq:entropy-scaling}
  H_\textup{ent}( \lambda A, d_\mathcal{H}, \varepsilon) = H_\textup{ent}( A, d_\mathcal{H}, \varepsilon \abs[0]{\lambda}^{-1})
\end{equation}
for any $\lambda \neq 0$; see, for example, Equation~(4.171) in~\citep{GineNickl2015}.

\begin{lemma} \label{lemma:metric-entropy-comparison}
  Let $k > d/2$.
  Suppose that $\mathcal{H}(K, \Omega) \simeq H^k(\Omega)$, that Assumption~\ref{assumption:regularity} holds, and that $c \leq 0$.
  Let $B_1^u$ and $B_1^k$ denote the unit balls of $(\mathcal{H}(K_u, \Omega), \norm[0]{\cdot}_{K_u})$ and $(H^k(\Omega), \norm[0]{\cdot}_{H^k(\Omega)})$, respectively.
  Then there is a constant $C = C(K, \Omega, \mathcal{L})$ such that
  \begin{equation} \label{eq:metric-entropy-comparison}
    H_\textup{ent}\big( B_1^u, \norm[0]{\cdot}_{L^\infty(\Omega)}, \varepsilon \big) \leq H_\textup{ent} \big( B_{1}^k, \norm[0]{\cdot}_{L^\infty(\Omega)}, C \varepsilon \big) .
  \end{equation}
\end{lemma}
\begin{proof}
  Fix $\varepsilon > 0$ and denote $n_\varepsilon = N( B_1^k, \norm[0]{\cdot}_{L^\infty(\Omega)}, \varepsilon)$ and let $f_1, \ldots, f_{n_\varepsilon} \in B_1^k$ be such that
  \begin{equation*}
    B_1^k = B_1\big( 0; H^k(\Omega), \lVert \cdot \rVert_{H^k(\Omega)} \big) \subset \bigcup_{i=1}^{n_\varepsilon} B_\varepsilon \big( f_i; H^k(\Omega), \lVert \cdot \rVert_{L^\infty(\Omega)} \big) .
  \end{equation*}
  Then
  \begin{equation} \label{eq:entropy-image-inclusion}
    \mathcal{L}^{-1} B_1\big( 0; H^k(\Omega), \lVert \cdot \rVert_{H^k(\Omega)} \big) \subset \bigcup_{i=1}^{n_\varepsilon} \mathcal{L}^{-1} B_\varepsilon \big( f_i; H^k(\Omega), \lVert \cdot \rVert_{L^\infty(\Omega)} \big) .
  \end{equation}
  We have $\norm[0]{u}_{K_u} = \norm[0]{f}_{K} \geq C_K \norm[0]{f}_{H^k(\Omega)}$ by Theorem~\ref{thm:Ku-RKHS}. 
  The norm-equivalence $\mathcal{H}(K, \Omega) \simeq H^k(\Omega)$ and the estimate $\norm[0]{u}_{L^\infty(\Omega)} \leq C_\infty \norm[0]{f}_{L^\infty(\Omega)}$ for $C_\infty = C(\Omega, \mathcal{L})$ follow from the Sobolev embedding theorem and Theorem~\ref{thm:pde-uniform-bound}.
  Therefore
  \begin{equation*}
    B_1^u = B_1\big(0; \mathcal{H}(K_u, \Omega), \norm[0]{\cdot}_{K_u} \big) = \mathcal{L}^{-1} B_1\big(0; \mathcal{H}(K, \Omega), \norm[0]{\cdot}_{K} \big) \subset \mathcal{L}^{-1} B_{1/C_K}\big(0; H^k(\Omega), \norm[0]{\cdot}_{H^k(\Omega)} \big)
  \end{equation*}
  and
  \begin{equation*}
    \mathcal{L}^{-1} B_{\varepsilon}\big(f_i; H^k(\Omega), \norm[0]{\cdot}_{L^\infty(\Omega)} \big) \subset B_{C_\infty \varepsilon}\big(u_i; \mathcal{H}(K_u, \Omega), \norm[0]{\cdot}_{L^\infty(\Omega)} \big),
  \end{equation*}
  where $u_i = \mathcal{L}^{-1} f_i$.
  Applying these two inclusion relations to~\eqref{eq:entropy-image-inclusion} and using the definition of metric entropy, together with~\eqref{eq:entropy-scaling}, yields the claim.
\end{proof}

\begin{proposition} \label{prop:small-ball}
  Let $k > d/2$.
  Suppose that $\mathcal{H}(K, \Omega) \simeq H^k(\Omega)$, that Assumption~\ref{assumption:regularity} holds, and that $c \leq 0$.
  Let $B_1^u$ denote the unit ball of $(\mathcal{H}(K_u, \Omega), \norm[0]{\cdot}_{K_u})$.
  Then there is a positive constant $C$, which does not depend on $\varepsilon$, such that 
  \begin{equation*}
    \beta(\varepsilon) \leq C \varepsilon^{-2d/(2k-d)}
  \end{equation*}
  for sufficiently small $\varepsilon$.
\end{proposition}
\begin{proof}
  It is a standard result that the metric entropy of the unit ball of $H^k(\Omega)$ in Lemma~\ref{lemma:metric-entropy-comparison} satisfies
  \begin{equation*}
    H_\textup{ent} \big( B_{1}^k, \norm[0]{\cdot}_{L^\infty(\Omega)}, \varepsilon \big) \leq C_\textup{ent} \varepsilon^{-d/k}
  \end{equation*}
  for a positive constant $C_\textup{ent} = C(k)$ and any $\varepsilon < 1$.
  See, for instance, Theorem~4.3.36 in~\citep{GineNickl2015}, Theorem~3.3.2 in~\citep{EdmundsTriebel1996}, the proof of Lemma~3 in~\citep{VaartZanten2011}, and Appendix~F in~\citep{Wynne2020}.
  It follows from Equation~\eqref{eq:metric-entropy-comparison} that
  \begin{equation} \label{eq:R-RKHS-entropy}
    H_\textup{ent}\big( B_1^u, \norm[0]{\cdot}_{L^\infty(\Omega)}, \varepsilon \big) \leq C_\textup{ent} C^{-d/k} \varepsilon^{-d/k}
  \end{equation}
  for sufficiently small $\varepsilon$.
  According to Theorem~1.2 in~\citep{LiLinde1999}, the estimate~\eqref{eq:R-RKHS-entropy} implies that
  \begin{equation*}
    \beta(\varepsilon) \leq C' \varepsilon^{-2d/(2k-d)}
  \end{equation*}
  for a positive constant $C'$ which does not depend on $\varepsilon$.
\end{proof}

A combination of Propositions~\ref{prop:decentering-term} and~\ref{prop:small-ball} yields an estimate for $\phi_{u_0}(\varepsilon)$.
Define the function
\begin{equation*}
  \psi_{u_0}(\varepsilon) = \frac{\phi_{u_0}(\varepsilon)}{\varepsilon^2}
\end{equation*}
and let $\psi_{u_0}^{-1}(n) = \sup \Set{ \varepsilon > 0}{\psi_{u_0}(\varepsilon) \geq n}$ denote its generalised inverse.

\begin{theorem} \label{thm:psi-inverse-theorem}
  Let $k \geq r > d/2$.
  Suppose that $\mathcal{H}(K, \Omega) \simeq H^k(\Omega)$, that Assumption~\ref{assumption:regularity} holds, and that $c \leq 0$.  
  If there exists $f_0 \in H^r(\R^d) \cap C^r(\R^d)$ such that $u_0 = \mathcal{L}^{-1} f_0 |_\Omega$, then there is a positive constant $C$, which does not depend on $\varepsilon$, such that
  \begin{equation*}
    \psi_{u_0}^{-1}(n) \leq C n^{- \min\{r, k - d/2\} / (2k)}
  \end{equation*}
  for all sufficiently large $n \geq 1$.
\end{theorem}
\begin{proof}
  By Propositions~\ref{prop:decentering-term} and~\ref{prop:small-ball},
  \begin{equation*}
    \psi_{u_0}(\varepsilon) \leq C_0 \big( \varepsilon^{-2(k-r)/r-2} + \varepsilon^{-2d/(2k-d)-2} \big) \leq 2 C_0 \varepsilon^{-2k / \min\{r, k - d/2\}}
  \end{equation*}
  whenever $\varepsilon$ is sufficiently small, where the positive constant $C_0$ does not depend on $\varepsilon$.
  It follows from the definition of $\psi_{u_0}^{-1}$ that
  \begin{equation*}
    \psi_{u_0}^{-1}(n) \leq C n^{- \min\{r, k - d/2\} / (2k)}
  \end{equation*}
  for $C = (2C_0)^{\min\{r, k - d/2\} / (2k)}$.
\end{proof}

\subsection{Proofs of main results} \label{sec:error-estimates}

We are now ready to prove the theorems in Section~\ref{sec:results}.
Given an $n$-vector $\b{Z}$, we employ the interpolation operator notation
\begin{equation} \label{eq:smoother-function}
  I_{X}(\b{Z})(\b{x}) = \b{K}_u(\b{x}, X)^\T ( \b{K}_u(X, X) + \sigma_\varepsilon^2 \b{I}_n)^{-1} \b{Z}.
\end{equation}
That is, the ideal conditional mean in~\eqref{eq:exact-conditional-mean} can be written as
\begin{equation} \label{eq:conditional-mean-decomposition}
  m_{u \mid \b{Y}} = m_u - I_{X}(\b{m}_u(X)) + I_{X}(\b{Y}).
\end{equation}

\begin{proof}[\normalfont\textbf{Proof of Theorem~\ref{theorem:error-exact}}]
  By Theorems~\ref{thm:pde-regularity} and~\ref{thm:Ku-RKHS}, $u_k, m_u \in \mathcal{H}(K_u, \Omega) \subset H^{k+2}(\Omega)$.
  Therefore it follows from Theorem~\ref{thm:arcangeli} with $p = 2$ and $g = u_t - m_{u \mid \b{Y}}$ that there is a constant $C_1$ independent of $X$ such that
  \begin{equation} \label{eq:estimate-exact-sampling-ineq}
    \norm[0]{u_t - m_{u \mid \b{Y}}}_{L^2(\Omega)} \leq C_1 \Big( h_{X,\Omega}^{k+2} \norm[0]{u_t - m_{u \mid \b{Y}}}_{H^{k+2}(\Omega)} + h_{X,\Omega}^{ d / 2} \norm[0]{\b{u}_t(X) - \b{m}_{u \mid \b{Y}}(X)}_2 \Big) .
  \end{equation}
  The decomposition in~\eqref{eq:conditional-mean-decomposition} gives
  \begin{equation} \label{eq:mean-decomposition-noisy}
    \norm[0]{u_t - m_{u \mid \b{Y}}}_{H^{k+2}(\Omega)} \leq \norm[0]{u_t - I_{X}(\b{Y})}_{H^{k+2}(\Omega)} + \norm[0]{m_u - I_{X}(\b{m}_u(X))}_{H^{k+2}(\Omega)}.
  \end{equation}
  The triangle inequality and Lemma~17 in~\citep{Wynne2020}, in combination with~\eqref{eq:Ku-norm-equivalence}, yield
  \begin{equation} \label{eq:estimate-exact-interm-1}
    \begin{split}
      \norm[0]{u_t - I_{X}(\b{Y})}_{H^{k+2}(\Omega)} \leq \norm[0]{u_t}_{H^{k+2}(\Omega)} + \norm[0]{ I_{X}(\b{Y}) }_{H^{k+2}(\Omega)} &\leq \norm[0]{u_t}_{H^{k+2}(\Omega)} + C_u^{-1} \norm[0]{ I_{X}(\b{Y}) }_{K_u}  \\
      &\leq (1 + C_u^{-1} C_u') \norm[0]{ u_t }_{H^{k+2}(\Omega)} + C_u^{-1} \sigma_\varepsilon^{-1} \norm[0]{\pmb{\varepsilon}}_2 \\
      &\leq 2 C_u^{-1} C_u' \norm[0]{ u_t }_{H^{k+2}(\Omega)} + C_u^{-1} \sigma_\varepsilon^{-1} \norm[0]{\pmb{\varepsilon}}_2
      \end{split}
  \end{equation}
  where $\pmb{\varepsilon} = (\varepsilon_1, \ldots, \varepsilon_n) \in \R^n$ is the noise vector and we used the fact that $C_u \leq C_u'$.
  The second term in~\eqref{eq:mean-decomposition-noisy} has the bound
  \begin{equation} \label{eq:estimate-exact-interm-11}
    \norm[0]{m_u - I_{X}(\b{m}_u(X))}_{H^{k+2}(\Omega)} \leq 2 C_u^{-1} C_u' \norm[0]{m_u}_{H^{k+2}(\Omega)},
  \end{equation}
  which is obtained in the same way as~\eqref{eq:estimate-exact-interm-1} but with $\pmb{\varepsilon}$ set as the zero vector.
  From Theorem~22 in~\citep{Wynne2020} and Theorem~\ref{thm:psi-inverse-theorem} with $r=k$ and $f_0 = f_t$ (i.e., $u_0 = u_t$ and $\min\{r, k - d/2\} = k - d/2$) we get
  \begin{equation} \label{eq:estimate-exact-interm-2}
    \mathbb{E} \big[ \norm[0]{\b{u}_t(X) - \b{m}_{u \mid \b{Y}}(X)}_2 \big] \leq C_2 \sqrt{n} \psi_{u_t}^{-1}(n) \leq C_2 C_3 n^{d/(4k)} ,
  \end{equation}
  for positive constants $C_2$ and $C_3$ which do not depend on $X$.
  Here we have enlarged the constant $C_3$ to remove the requirement that $n$ be sufficiently large.
  Inserting the estimates~\eqref{eq:estimate-exact-interm-1}--\eqref{eq:estimate-exact-interm-2} into~\eqref{eq:estimate-exact-sampling-ineq} and using the bound $\sigma_\varepsilon^{-1} \mathbb{E} [ \norm[0]{\pmb{\varepsilon}}_2] \leq \sqrt{n}$, which follows from the Gaussianity of the noise terms, yields
    \begin{equation} \label{eq:main-proof-final-eq}
      \begin{split}
        \mathbb{E} \big[ \norm[0]{u_t - m_{u \mid \b{Y}}}_{L^2(\Omega)} \big] \leq{}& 2 C_1 C_u^{-1} C_u' h_{X,\Omega}^{k+2} \big( \norm[0]{ u_t }_{H^{k+2}(\Omega)} + \norm[0]{m_u}_{H^{k+2}(\Omega)} \big) + C_1 C_u^{-1} h_{X,\Omega}^{k+2} \sigma_\varepsilon^{-1} \mathbb{E} [ \norm[0]{\pmb{\varepsilon}}_2] + C_1 C_2 C_3 h_{X,\Omega}^{d/2} n^{d/(4k)}  \\
        \leq{}& 2 C_1 C_u^{-1} C_u' h_{X,\Omega}^{k+2} \big( \norm[0]{ u_t }_{H^{k+2}(\Omega)} + \norm[0]{m_u}_{H^{k+2}(\Omega)} \big) + C_1 C_u^{-1} h_{X,\Omega}^{k+2} \sqrt{n} + C_1 C_2 C_3 h_{X,\Omega}^{d/2} n^{d/(4k)}.
      \end{split}
    \end{equation}
    This concludes the proof of~\eqref{eq:exact-posterior-bound-not-quasi-uniform}.
\end{proof}

\begin{proof}[\normalfont\textbf{Proof of Theorem~\ref{theorem:error-exact-noiseless}}]
  The proof proceeds exactly as that of Theorem~\ref{theorem:error-exact}, except that now $\pmb{\varepsilon} = \b{0}$ and \smash{$\b{m}_{u \mid \b{Y}}(X) = \b{u}_t(X)$} because $\b{Y} = \b{u}_t(X)$. This means that the terms in~\eqref{eq:main-proof-final-eq} that arise from $\lVert \pmb{\varepsilon} \rVert_2$ and~\eqref{eq:estimate-exact-interm-2} are now zero, so that we are left with the claimed bound.
\end{proof}

The following proposition allows us to make use of Assumption~\ref{assumption:fe-error} on the error of the finite element discretisation.
Although this basic proposition must have appeared several times and in various forms in the literature on scalable approximations for GP regression, we have not been able to locate a convenient reference for it.

\begin{proposition} \label{prop:kernel-error}
  Let $R_1$ and $R_2$ be any positive-semidefinite kernels, $\sigma > 0$, and $\b{Z} \in \R^n$.
  If
  \begin{equation*}
    \sup_{\b{x}, \b{y} \in \Omega} \abs[0]{ R_1(\b{x}, \b{y}) - R_2(\b{x}, \b{y}) } = \delta
  \end{equation*}
  for some $\delta > 0$, then the functions
  \begin{align*}
    m_{\mid \b{Z}}^1(\b{x}) = \b{R}_1(\b{x}, X) \big( \b{R}_1(X, X) + \sigma^2 \b{I}_n \big)^{-1} \b{Z}, \quad\quad 
    m_{\mid \b{Z}}^2(\b{x}) = \b{R}_2(\b{x}, X) \big( \b{R}_2(X, X) + \sigma^2 \b{I}_n \big)^{-1} \b{Z}
  \end{align*}
  satisfy
  \begin{equation} \label{eq:kernel-error}
    \sup_{ \b{x} \in \Omega} \abs[0]{ m_{\mid \b{Z}}^1(\b{x}) - m_{\mid \b{Z}}^2(\b{x}) } \leq \norm[0]{ \b{Z} }_2 (\delta + C \sigma^{-2}) \sigma^{-2} \delta n,
  \end{equation}
  where $C = \sup_{\b{x}, \b{y} \in \Omega} \abs[0]{R_2(\b{x}, \b{y})}$.
\end{proposition}
\begin{proof} Write
  \begin{equation*}
    \begin{split}
      \lvert m_{\mid \b{Z}}^1(\b{x}) - m_{\mid \b{Z}}^2(\b{x}) \rvert &= \abs[1]{\b{Z}^\T [ ( \b{R}_1(X, X) + \sigma^2 \b{I}_n )^{-1} \b{R}_1(\b{x}, X) - ( \b{R}_2(X, X) + \sigma^2 \b{I}_n )^{-1} \b{R}_2(\b{x}, X) ] } \\
      &\leq \norm[0]{ \b{Z} }_2 \norm[1]{ (\b{R}_1(X, X) + \sigma^2 \b{I}_n)^{-1} \b{R}_1(\b{x}, X) - ( \b{R}_2(X, X) + \sigma^2 \b{I}_n )^{-1} \b{R}_2(\b{x}, X) }_2.
      \end{split}
  \end{equation*}
  Let $R(\b{x}, \b{y}) = R_1(\b{x}, \b{y}) - R_2(\b{x}, \b{y})$ so that $\sup_{ \b{x}, \b{y} \in \Omega} \abs[0]{ R(\b{x}, \b{y}) } = \delta$ and 
  \begin{equation*}
    \begin{split}
      \big\lVert ( \b{R}_1(X, &X) + \sigma^2 \b{I}_n )^{-1} \b{R}_1(\b{x}, X) - ( \b{R}_2(X, X) + \sigma^2 \b{I}_n )^{-1} \b{R}_2(\b{x}, X) \big\rVert_2 \\
      ={}& \big\lVert ( \b{R}_2(X, X) + \b{R}(X, X) + \sigma^2 \b{I}_n )^{-1} (\b{R}(\b{x}, X) + \b{R}_2(\b{x}, X)) - ( \b{R}_2(X, X) + \sigma^2 \b{I}_n )^{-1} \b{R}_2(\b{x}, X) \big\rVert_2 \\
      \leq{}& \norm[1]{ [ ( \b{R}_2(X, X) + \b{R}(X, X) + \sigma^2 \b{I}_n )^{-1} - ( \b{R}_2(X, X) + \sigma^2 \b{I}_n )^{-1} ] \b{R}_2(\b{x}, X)}_2 + \norm[1]{ ( \b{R}_2(X, X) + \b{R}(X, X) + \sigma^2 \b{I}_n )^{-1} \b{R}(\b{x}, X)}_2 \\
      \leq{}& \sqrt{n} C \norm[1]{ [ ( \b{R}_2(X, X) + \b{R}(X, X) + \sigma^2 \b{I}_n )^{-1} - (\b{R}_2(X, X) + \sigma^2 \b{I}_n)^{-1} }_2 \\
        &+ \sqrt{n} \delta \norm[1]{ ( \b{R}_2(X, X) + \b{R}(X, X) + \sigma^2 \b{I}_n )^{-1}}_2.
        \end{split}
  \end{equation*}
  Because the matrix $\b{R}_1(X, X) = \b{R}_2(X, X) + \b{R}(X, X)$ is positive-semidefinite, the largest singular value of the matrix $(\b{R}_2(X, X) + \b{R}(X, X) + \sigma^2 \b{I}_n)^{-1}$ is $(\sigma^2 + \lambda_\textup{min}(\b{R}_1(X,X)))^{-1}$.
  Therefore
  \begin{equation*}
    \sqrt{n} \delta \norm[1]{ ( \b{R}_2(X, X) + \b{R}(X, X) + \sigma^2 \b{I}_n )^{-1}}_2 = \sqrt{n} \delta [\sigma^2 + \lambda_\textup{min}(\b{R}_1(X,X)) ]^{-1} \leq \sqrt{n} \sigma^{-2} \delta.
  \end{equation*}
  Finally, 
  \begin{equation*}
    \begin{split}
      \big\lVert ( \b{R}_2(X, X) + \b{R}(X, &X) + \sigma^2 \b{I}_n )^{-1} - ( \b{R}_2(X, X) + \sigma^2 \b{I}_n )^{-1} \big\rVert_2 \\
      &= \big\lVert ( \b{R}_2(X, X) + \b{R}(X, X) + \sigma^2 \b{I}_n )^{-1} \b{R}(X, X) ( \b{R}_2(X, X) + \sigma^2 \b{I}_n )^{-1} \big\rVert_2 \\
      &\leq \lVert \b{R}(X, X) \rVert_2 \big\lVert ( \b{R}_2(X, X) + \b{R}(X, X) + \sigma^2 \b{I}_n )^{-1} \big\rVert_2 \big\lVert ( \b{R}_2(X, X) + \sigma^2 \b{I}_n )^{-1} \big\rVert_2 \\
      &\leq \sqrt{n} \sigma^{-4} \delta.
      \end{split}
  \end{equation*}
  The claim follows by putting these estimates together.
\end{proof}

The proof of Theorem~\ref{theorem:error-fe} is a straightforward combination of Theorem~\ref{theorem:error-exact} and Proposition~\ref{prop:kernel-error}.

\begin{proof}[\normalfont\textbf{Proof of Theorem~\ref{theorem:error-fe}}]
  The triangle inequality yields
  \begin{equation} \label{eq:error-fe-combined-term}
    \norm[0]{ u_t - m_{u \mid \b{Y}}^\FE }_{L^2(\Omega)} \leq \norm[0]{ u_t - m_{u \mid \b{Y}} }_{L^2(\Omega)} + \norm[0]{ m_{u \mid \b{Y}} - m_{u \mid \b{Y}}^\FE }_{L^2(\Omega)}.
  \end{equation}
  Theorem~\ref{theorem:error-exact} bounds the expectation of the first term as
  \begin{equation} \label{eq:error-fe-1nd-term-bound}
    \mathbb{E} \big[ \norm[0]{ u_t - m_{u \mid \b{Y}} }_{L^2(\Omega)} \big] \leq C_1 n^{-1/2 + d/(4k)}
  \end{equation}
  for a constant $C_1 > 0$ that is independent of $X$, while Proposition~\ref{prop:kernel-error} and Assumption~\ref{assumption:fe-error} give
  \begin{equation*}
    \norm[0]{ m_{u \mid \b{Y}} - m_{u \mid \b{Y}}^\FE }_{L^2(\Omega)} \leq C_2 \norm[0]{\b{Y}}_2 ( n_\FE^{-q} + \sigma_\varepsilon^{-2} ) \sigma_\varepsilon^{-2} n_\FE^{-q} n
  \end{equation*}
  for a constant $C_2 > 0$ that is independent of $X$.
  Since
  \begin{equation*}
    \mathbb{E}[ \norm[0]{\b{Y}}_2 ] \leq \norm[0]{\b{u}_t(X)}_2 + \mathbb{E}[ \norm[0]{\pmb{\varepsilon}}_2] \leq ( \norm[0]{u_t}_{L^\infty(\Omega)} + \sigma_\varepsilon ) \sqrt{n},
  \end{equation*}
  from Theorem~\ref{thm:pde-uniform-bound} we obtain 
  \begin{equation} \label{eq:error-fe-2nd-term-bound}
    \mathbb{E} \big[ \norm[0]{ m_{u \mid \b{Y}} - m_{u \mid \b{Y}}^\FE }_{L^2(\Omega)} \big] \leq C_3 ( n_\FE^{-q} + \sigma_\varepsilon^2) \sigma_\varepsilon^{-2} ( \norm[0]{f_t}_{L^\infty(\Omega)} + \sigma_\varepsilon ) n_\FE^{-q} n^{3/2}
  \end{equation}
  for a constant $C_3 > 0$ that is independent of $X$.
  Taking expectation of~\eqref{eq:error-fe-combined-term} and using the bounds~\eqref{eq:error-fe-1nd-term-bound} and~\eqref{eq:error-fe-2nd-term-bound} concludes the proof.
\end{proof}

\begin{proof}[\normalfont\textbf{Proof of Theorem~\ref{thm:error-with-discrepancy}}]
  By Theorem~\ref{thm:Ku-RKHS}, the norm-equivalence assumption, and the inequality $k_1 + 2 \geq r$, it holds that $\mathcal{H}(K_u, \Omega) \subset \mathcal \mathcal{H}(K_d, \Omega)$.
  From this inclusion, Theorem~\ref{thm:rkhs-sum} and~\eqref{eq:Ku-norm-equivalence} it follows that $\mathcal{H}(K_{ud}, \Omega) \simeq H^r(\Omega)$.
  By Theorem~\ref{thm:Ku-RKHS} and our assumptions, the functions $m_d$, $m_u$, and $u_t$ are in $H^{k_2 + 2}(\Omega)$ and $r \geq k_2 + 2$.
  We can therefore apply Theorem~2 in~\citep{Wynne2020} with
  \begin{equation*}
    k = K_{ud}, \quad f = u_t, \quad \tau_k^- = \tau_k^+ = r, \quad \tau_f = k_2 + 2, \quad s = 0 , \quad q = 2.
  \end{equation*}
  This yields the estimate 
  \begin{equation*}
      \mathbb{E} \big[ \norm[0]{ u_t - m_{d;u \mid \b{Y}} }_{L^2(\Omega)} \big] \leq C h_{X,\Omega}^{d/2} \big( h_{X,\Omega}^{k_2 + 2 - d/2} \rho_{X,\Omega}^{r - k_2 - 2} + \sqrt{n} h_{X,\Omega}^{r - d/2} + n^\kappa \big) = C \big( h_{X,\Omega}^{k_2 + 2} \rho_{X,\Omega}^{r - k_2 - 2} + \sqrt{n} \, h_{X,\Omega}^{r} + n^{\kappa(k_2, r)} h_{X,\Omega}^{d/2} \big),
  \end{equation*}
  where
  \begin{equation} \label{eq:kappa-def}
    \kappa(k_2, r) = \max\bigg\{ \frac{1}{2} - \frac{k_2 + 2}{2r}, \frac{d}{4(k_2 + 2)} \bigg\} \leq \frac{1}{2},
  \end{equation}
  for a positive constant $C$ that is independent of $X$.
  Theorem~2 in~\citep{Wynne2020} requires that $h_{X,\Omega}$ be sufficiently small. We eliminate this assumption by enlarging $C$.
  This proves~\eqref{eq:exact-posterior-disc-bound-not-quasi-uniform} while~\eqref{eq:exact-posterior-disc-bound-quasi-uniform} follows from $h_{X,\Omega} = O(n^{-1/d})$ and the fact that the mesh ratio $\rho_{X,\Omega}$ is bounded for quasi-uniform points.
\end{proof}

\begin{proof}[\normalfont\textbf{Proof of Theorem~\ref{thm:error-fe-with-discrepancy}}]
  The proof is identical to that of Theorem~\ref{theorem:error-fe} expect that the bound~\eqref{eq:exact-posterior-disc-bound-quasi-uniform} is used in place of~\eqref{eq:exact-posterior-bound-quasi-uniform}.
\end{proof}

\section*{Acknowledgements}

TK was supported by the Research Council of Finland grants 338567 (``Scalable, adaptive and reliable probabilistic integration'') and 359183 (``Flagship of Advanced Mathematics for Sensing, Imaging and Modelling'').
We thank the associate editor and reviewers for comments and suggestions that improved the article.


\begin{thebibliography}{49}
\expandafter\ifx\csname natexlab\endcsname\relax\def\natexlab#1{#1}\fi
\providecommand{\bibinfo}[2]{#2}
\ifx\xfnm\relax \def\xfnm[#1]{\unskip,\space#1}\fi
\bibitem[{Abdulle and Garegnani(2021)}]{Abdulle2021}
\bibinfo{author}{A.~Abdulle}, \bibinfo{author}{G.~Garegnani}, \bibinfo{title}{A probabilistic finite element method based on random meshes: {E}rror estimators and {B}ayesian inverse problems}, \bibinfo{journal}{Computer Methods in Applied Mechanics and Engineering} \bibinfo{volume}{384} (\bibinfo{year}{2021}) \bibinfo{pages}{113961}.
\bibitem[{Arcang{\'e}li et~al.(2007)Arcang{\'e}li, de~Silanes and Torrens}]{Arcangeli2007}
\bibinfo{author}{R.~Arcang{\'e}li}, \bibinfo{author}{M.~C.~L. de~Silanes}, \bibinfo{author}{J.~J. Torrens}, \bibinfo{title}{An extension of a bound for functions in {S}obolev spaces, with applications to $(m,s)$-spline interpolation and smoothing}, \bibinfo{journal}{Numerische Mathematik} \bibinfo{volume}{107} (\bibinfo{year}{2007}) \bibinfo{pages}{181--211}.
\bibitem[{Berlinet and Thomas{-}Agnan(2004)}]{Berlinet2004}
\bibinfo{author}{A.~Berlinet}, \bibinfo{author}{C.~Thomas{-}Agnan}, \bibinfo{title}{Reproducing Kernel {H}ilbert Spaces in Probability and Statistics}, \bibinfo{publisher}{Springer}, \bibinfo{address}{Berlin}, \bibinfo{year}{2004}.
\bibitem[{Brenner and Scott(2008)}]{BrennerScott2008}
\bibinfo{author}{S.~C. Brenner}, \bibinfo{author}{L.~R. Scott}, \bibinfo{title}{The Mathematical Theory of Finite Element Methods}, number~\bibinfo{number}{15} in \bibinfo{series}{Texts in Applied Mathematics}, \bibinfo{publisher}{Springer}, \bibinfo{address}{Berlin}, \bibinfo{year}{2008}.
\bibitem[{Briol et~al.(2019)Briol, Oates, Girolami, Osborne and Sejdinovic}]{Briol2019}
\bibinfo{author}{F.-X. Briol}, \bibinfo{author}{C.~J. Oates}, \bibinfo{author}{M.~Girolami}, \bibinfo{author}{M.~A. Osborne}, \bibinfo{author}{D.~Sejdinovic}, \bibinfo{title}{Probabilistic integration: A role in statistical computation? (with discussion and rejoinder)}, \bibinfo{journal}{Statistical Science} \bibinfo{volume}{34} (\bibinfo{year}{2019}) \bibinfo{pages}{1--22}.
\bibitem[{Cialenco et~al.(2012)Cialenco, Fasshauer and Ye}]{Cialenco2012}
\bibinfo{author}{I.~Cialenco}, \bibinfo{author}{G.~E. Fasshauer}, \bibinfo{author}{Q.~Ye}, \bibinfo{title}{Approximation of stochastic partial differential equations by a kernel-based collocation method}, \bibinfo{journal}{International Journal of Computer Mathematics} \bibinfo{volume}{89} (\bibinfo{year}{2012}) \bibinfo{pages}{2543--2561}.
\bibitem[{Cockayne et~al.(2017)Cockayne, Oates, Sullivan and Girolami}]{Cockayne2017}
\bibinfo{author}{J.~Cockayne}, \bibinfo{author}{C.~J. Oates}, \bibinfo{author}{T.~Sullivan}, \bibinfo{author}{M.~Girolami}, \bibinfo{title}{Probabilistic numerical methods for partial differential equations and {B}ayesian inverse problems}, in: \bibinfo{booktitle}{Proceedings of the 36th International Workshop on Bayesian Inference and Maximum Entropy Methods in Science and Engineering}, volume \bibinfo{volume}{1853} of \text{\bibinfo{series}{AIP Conference Proceedings}}.
\bibitem[{Duffin et~al.(2021)Duffin, Cripps, Stemler and Girolami}]{Duffin2021}
\bibinfo{author}{C.~Duffin}, \bibinfo{author}{E.~Cripps}, \bibinfo{author}{T.~Stemler}, \bibinfo{author}{M.~Girolami}, \bibinfo{title}{Statistical finite elements for misspecified models}, \bibinfo{journal}{Proceedings of the National Academy of Sciences} \bibinfo{volume}{118} (\bibinfo{year}{2021}).
\bibitem[{Edmunds and Triebel(1996)}]{EdmundsTriebel1996}
\bibinfo{author}{D.~Edmunds}, \bibinfo{author}{H.~Triebel}, \bibinfo{title}{Function Spaces, Entropy Numbers, Differential Operators}, \bibinfo{publisher}{Cambridge University Press}, \bibinfo{year}{1996}.
\bibitem[{Evans(1998)}]{Evans1998}
\bibinfo{author}{L.~C. Evans}, \bibinfo{title}{Partial Differential Equations}, number~\bibinfo{number}{19} in \bibinfo{series}{Graduate Studies in Mathematics}, \bibinfo{publisher}{American Mathematical Society}, \bibinfo{address}{New York}, \bibinfo{year}{1998}.
\bibitem[{Fasshauer(1996)}]{Fasshauer1996}
\bibinfo{author}{G.~E. Fasshauer}, \bibinfo{title}{Solving partial differential equations by collocation with radial basis functions}, in: \bibinfo{booktitle}{Proceedings of Chamonix}, \bibinfo{publisher}{Vanderbilt University Press}, \bibinfo{year}{1996}, pp. \bibinfo{pages}{1--8}.
\bibitem[{Fasshauer and Ye(2011)}]{FasshauerYe2011}
\bibinfo{author}{G.~E. Fasshauer}, \bibinfo{author}{Q.~Ye}, \bibinfo{title}{Reproducing kernels of generalized {S}obolev spaces via a {G}reen function approach with distributional operators}, \bibinfo{journal}{Numerische Mathematik} \bibinfo{volume}{119} (\bibinfo{year}{2011}) \bibinfo{pages}{585--611}.
\bibitem[{Fasshauer and Ye(2013)}]{FasshauerYe2013}
\bibinfo{author}{G.~E. Fasshauer}, \bibinfo{author}{Q.~Ye}, \bibinfo{title}{Reproducing kernels of {S}obolev spaces via a {G}reen kernel approach with differential operators and boundary operators}, \bibinfo{journal}{Advances in Computational Mathematics} \bibinfo{volume}{38} (\bibinfo{year}{2013}) \bibinfo{pages}{891--921}.
\bibitem[{Febrianto et~al.(2022)Febrianto, Butler, Girolami and Cirak}]{febrianto:2021}
\bibinfo{author}{E.~Febrianto}, \bibinfo{author}{L.~Butler}, \bibinfo{author}{M.~Girolami}, \bibinfo{author}{F.~Cirak}, \bibinfo{title}{Digital twinning of self-sensing structures using the statistical finite element method}, \bibinfo{journal}{Data-Centric Engineering} \bibinfo{volume}{3} (\bibinfo{year}{2022}) \bibinfo{pages}{e31}.
\bibitem[{Franke and Schaback(1998)}]{FrankeSchaback1998}
\bibinfo{author}{C.~Franke}, \bibinfo{author}{R.~Schaback}, \bibinfo{title}{Solving partial differential equations by collocation using radial basis functions}, \bibinfo{journal}{Applied Mathematics and Computation} \bibinfo{volume}{93} (\bibinfo{year}{1998}) \bibinfo{pages}{73--82}.
\bibitem[{Ghosal et~al.(2000)Ghosal, Ghosh and van~der Vaart}]{Ghosal2000}
\bibinfo{author}{S.~Ghosal}, \bibinfo{author}{J.~L. Ghosh}, \bibinfo{author}{A.~van~der Vaart}, \bibinfo{title}{Convergence rates of posterior distributions}, \bibinfo{journal}{The Annals of Statistics} \bibinfo{volume}{28} (\bibinfo{year}{2000}) \bibinfo{pages}{500--531}.
\bibitem[{Gilbarg and Trudinger(1983)}]{GilbargTrudinger1983}
\bibinfo{author}{D.~Gilbarg}, \bibinfo{author}{N.~S. Trudinger}, \bibinfo{title}{Elliptic Partial Differential Equations of Second Order}, \bibinfo{publisher}{Springer}, \bibinfo{address}{Berlin}, \bibinfo{edition}{2nd} edition, \bibinfo{year}{1983}.
\bibitem[{Gin{\'e} and Nickl(2015)}]{GineNickl2015}
\bibinfo{author}{E.~Gin{\'e}}, \bibinfo{author}{R.~Nickl}, \bibinfo{title}{Mathematical Foundations of Infinite-Dimensional Statistical Models}, number~\bibinfo{number}{40} in \bibinfo{series}{Cambridge Series in Statistical and Probabilistic Mathematics}, \bibinfo{publisher}{Cambridge University Press}, \bibinfo{year}{2015}.
\bibitem[{Girolami et~al.(2021)Girolami, Febrianto, Yin and Cirak}]{Girolami2021}
\bibinfo{author}{M.~Girolami}, \bibinfo{author}{E.~Febrianto}, \bibinfo{author}{G.~Yin}, \bibinfo{author}{F.~Cirak}, \bibinfo{title}{The statistical finite element method (stat{FEM}) for coherent synthesis of observation data and model predictions}, \bibinfo{journal}{Computer Methods in Applied Mechanics and Engineering} \bibinfo{volume}{375} (\bibinfo{year}{2021}) \bibinfo{pages}{113533}.
\bibitem[{Graepel(2003)}]{Graepel2003}
\bibinfo{author}{T.~Graepel}, \bibinfo{title}{Solving noisy linear operator equations by {G}aussian processes: application to ordinary and partial differential equations}, in: \bibinfo{booktitle}{Proceedings of the 20th International Conference on International Conference on Machine Learning}, pp. \bibinfo{pages}{234--241}.
\bibitem[{Kanagawa et~al.(2020)Kanagawa, Sriperumbudur and Fukumizu}]{Kanagawa2020}
\bibinfo{author}{M.~Kanagawa}, \bibinfo{author}{B.~K. Sriperumbudur}, \bibinfo{author}{K.~Fukumizu}, \bibinfo{title}{Convergence analysis of deterministic kernel-based quadrature rules in misspecified settings}, \bibinfo{journal}{Foundations of Computational Mathematics} \bibinfo{volume}{20} (\bibinfo{year}{2020}) \bibinfo{pages}{155--194}.
\bibitem[{Kansa(1990)}]{Kansa1990}
\bibinfo{author}{E.~J. Kansa}, \bibinfo{title}{Multiquadrics---{A} scattered data approximation scheme with applications to computational fluid-dynamics---{II} solutions to parabolic, hyperbolic and elliptic partial differential equations}, \bibinfo{journal}{Computers \& Mathematics with Applications} \bibinfo{volume}{19} (\bibinfo{year}{1990}) \bibinfo{pages}{147--161}.
\bibitem[{Karvonen(2023)}]{Karvonen2023}
\bibinfo{author}{T.~Karvonen}, \bibinfo{title}{Asymptotic bounds for smoothness parameter estimates in {G}aussian process interpolation}, \bibinfo{journal}{SIAM/ASA Journal on Uncertainty Quantification} \bibinfo{volume}{11} (\bibinfo{year}{2023}) \bibinfo{pages}{1225--1257}.
\bibitem[{Karvonen et~al.(2020)Karvonen, Wynne, Tronarp, Oates and S{\"a}rkk{\"a}}]{Karvonen2020}
\bibinfo{author}{T.~Karvonen}, \bibinfo{author}{G.~Wynne}, \bibinfo{author}{F.~Tronarp}, \bibinfo{author}{C.~J. Oates}, \bibinfo{author}{S.~S{\"a}rkk{\"a}}, \bibinfo{title}{Maximum likelihood estimation and uncertainty quantification for {G}aussian process approximation of deterministic functions}, \bibinfo{journal}{SIAM/ASA Journal on Uncertainty Quantification} \bibinfo{volume}{8} (\bibinfo{year}{2020}) \bibinfo{pages}{926--958}.
\bibitem[{Kennedy and O'Hagan(2002)}]{KennedyOHagan2002}
\bibinfo{author}{M.~C. Kennedy}, \bibinfo{author}{A.~O'Hagan}, \bibinfo{title}{{B}ayesian calibration of computer models}, \bibinfo{journal}{Journal of the Royal Statistical Society: Series B} \bibinfo{volume}{63} (\bibinfo{year}{2002}) \bibinfo{pages}{425--464}.
\bibitem[{Knabner and Angermann(2021)}]{KnabnerAngermann2021}
\bibinfo{author}{P.~Knabner}, \bibinfo{author}{L.~Angermann}, \bibinfo{title}{Numerical Methods for Elliptic and Parabolic Partial Differential Equations. With Contributions by {A}ndreas {R}upp}, number~\bibinfo{number}{44} in \bibinfo{series}{Texts in Applied Mathematics}, \bibinfo{publisher}{Springer}, \bibinfo{address}{Berlin}, \bibinfo{year}{2021}.
\bibitem[{Koh and Cirak(2023)}]{koh2023stochastic}
\bibinfo{author}{K.~J. Koh}, \bibinfo{author}{F.~Cirak}, \bibinfo{title}{Stochastic {PDE} representation of random fields for large-scale {G}aussian process regression and statistical finite element analysis}, \bibinfo{journal}{Computer Methods in Applied Mechanics and Engineering} \bibinfo{volume}{417} (\bibinfo{year}{2023}) \bibinfo{pages}{116358}.
\bibitem[{Krieg and Sonnleitner(2024)}]{KriegSonnleitner2024}
\bibinfo{author}{D.~Krieg}, \bibinfo{author}{M.~Sonnleitner}, \bibinfo{title}{Random points are optimal for the approximation of {S}obolev functions}, \bibinfo{journal}{IMA Journal of Numerical Analysis} \bibinfo{volume}{44} (\bibinfo{year}{2024}) \bibinfo{pages}{1346--1371}.
\bibitem[{Li and Linde(1999)}]{LiLinde1999}
\bibinfo{author}{W.~V. Li}, \bibinfo{author}{W.~Linde}, \bibinfo{title}{Approximation, metric entropy and small ball estimates for {G}aussian measures}, \bibinfo{journal}{The Annals of Probability} \bibinfo{volume}{27} (\bibinfo{year}{1999}) \bibinfo{pages}{1556--1578}.
\bibitem[{Lord et~al.(2014)Lord, Powell and Shardlow}]{lord2014introduction}
\bibinfo{author}{G.~J. Lord}, \bibinfo{author}{C.~E. Powell}, \bibinfo{author}{T.~Shardlow}, \bibinfo{title}{An Introduction to Computational Stochastic {PDE}s}, \bibinfo{publisher}{Cambridge University Press}, \bibinfo{year}{2014}.
\bibitem[{Nitsche(1977)}]{Nitsche1977}
\bibinfo{author}{J.~Nitsche}, \bibinfo{title}{{$L_\infty$}-convergence of finite element approximations}, in: \bibinfo{booktitle}{Mathematical Aspects of Finite Element Methods}, number \bibinfo{number}{606} in \bibinfo{series}{Lecture Notes in Mathematics}, \bibinfo{publisher}{Springer-Verlag}, \bibinfo{year}{1977}, pp. \bibinfo{pages}{261--274}.
\bibitem[{Owhadi(2015)}]{Owhadi2015}
\bibinfo{author}{H.~Owhadi}, \bibinfo{title}{{B}ayesian numerical homogenization}, \bibinfo{journal}{Multiscale Modeling \& Simulation} \bibinfo{volume}{13} (\bibinfo{year}{2015}) \bibinfo{pages}{812--828}.
\bibitem[{Papandreou et~al.(2023)Papandreou, Cockayne, Girolami and Duncan}]{Papandreou2021}
\bibinfo{author}{Y.~Papandreou}, \bibinfo{author}{J.~Cockayne}, \bibinfo{author}{M.~Girolami}, \bibinfo{author}{A.~B. Duncan}, \bibinfo{title}{Theoretical guarantees for the statistical finite element method}, \bibinfo{journal}{SIAM/ASA Journal on Uncertainty Quantification} \bibinfo{volume}{11} (\bibinfo{year}{2023}) \bibinfo{pages}{1278--1307}.
\bibitem[{Paulsen and Raghupathi(2016)}]{Paulsen2016}
\bibinfo{author}{V.~I. Paulsen}, \bibinfo{author}{M.~Raghupathi}, \bibinfo{title}{An Introduction to the Theory of Reproducing Kernel {H}ilbert Spaces}, number \bibinfo{number}{152} in \bibinfo{series}{Cambridge Studies in Advanced Mathematics}, \bibinfo{publisher}{Cambridge University Press}, \bibinfo{year}{2016}.
\bibitem[{Raissi et~al.(2017)Raissi, Perdikaris and Karniadakis}]{Raissi2017}
\bibinfo{author}{M.~Raissi}, \bibinfo{author}{P.~Perdikaris}, \bibinfo{author}{G.~E. Karniadakis}, \bibinfo{title}{Machine learning of linear differential equations using {G}aussian processes}, \bibinfo{journal}{Journal of Computational Physics} \bibinfo{volume}{348} (\bibinfo{year}{2017}) \bibinfo{pages}{683--693}.
\bibitem[{Rasmussen and Williams(2006)}]{williams2006gaussian}
\bibinfo{author}{C.~E. Rasmussen}, \bibinfo{author}{C.~K.~I. Williams}, \bibinfo{title}{Gaussian Processes for Machine Learning}, Adaptive Computation and Machine Learning, \bibinfo{publisher}{MIT Press}, \bibinfo{address}{Boston}, \bibinfo{year}{2006}.
\bibitem[{Schatz and Wahlbin(1977)}]{SchatzWahlbin1977}
\bibinfo{author}{A.~H. Schatz}, \bibinfo{author}{L.~B. Wahlbin}, \bibinfo{title}{Interior maximum norm estimates for finite element methods}, \bibinfo{journal}{Mathematics of Computation} \bibinfo{volume}{31} (\bibinfo{year}{1977}) \bibinfo{pages}{414--442}.
\bibitem[{Schatz and Wahlbin(1978)}]{SchatzWahlbin1978}
\bibinfo{author}{A.~H. Schatz}, \bibinfo{author}{L.~B. Wahlbin}, \bibinfo{title}{Maximum norm estimates in the finite element method on plane polygonal domains. {I}}, \bibinfo{journal}{Mathematics of Computation} \bibinfo{volume}{32} (\bibinfo{year}{1978}) \bibinfo{pages}{73--109}.
\bibitem[{Schatz and Wahlbin(1995)}]{SchatzWahlbin1995}
\bibinfo{author}{A.~H. Schatz}, \bibinfo{author}{L.~B. Wahlbin}, \bibinfo{title}{Interior maximum norm estimates for finite element methods. {II}}, \bibinfo{journal}{Mathematics of Computation} \bibinfo{volume}{64} (\bibinfo{year}{1995}) \bibinfo{pages}{907--928}.
\bibitem[{Szab\'{o} et~al.(2015)Szab\'{o}, van~der Vaart and van Zanten}]{Szabo2015}
\bibinfo{author}{B.~Szab\'{o}}, \bibinfo{author}{A.~W. van~der Vaart}, \bibinfo{author}{J.~H. van Zanten}, \bibinfo{title}{Frequentist coverage of adaptive nonparametric {B}ayesian credible sets}, \bibinfo{journal}{The Annals of Statistics} \bibinfo{volume}{43} (\bibinfo{year}{2015}) \bibinfo{pages}{1391--1428}.
\bibitem[{Teckentrup(2020)}]{Teckentrup2020}
\bibinfo{author}{A.~L. Teckentrup}, \bibinfo{title}{Convergence of {G}aussian process regression with estimated hyper-parameters and applications in {B}ayesian inverse problems}, \bibinfo{journal}{SIAM/ASA Journal on Uncertainty Quantification} \bibinfo{volume}{8} (\bibinfo{year}{2020}) \bibinfo{pages}{1310--1337}.
\bibitem[{Travelletti and Ginsbourger(2024)}]{Travelletti2024}
\bibinfo{author}{C.~Travelletti}, \bibinfo{author}{D.~Ginsbourger}, \bibinfo{title}{Disintegration of {G}aussian measures for sequential assimilation of linear operator data}, \bibinfo{journal}{Electronic Journal of Statistics} \bibinfo{volume}{18} (\bibinfo{year}{2024}) \bibinfo{pages}{3825--3857}.
\bibitem[{Tsybakov(2009)}]{Tsybakov2009}
\bibinfo{author}{A.~B. Tsybakov}, \bibinfo{title}{Introduction to Nonparametric Estimation}, Springer Series in Statistics, \bibinfo{publisher}{Springer}, \bibinfo{address}{Berlin}, \bibinfo{year}{2009}.
\bibitem[{van~der Vaart and van Zanten(2011)}]{VaartZanten2011}
\bibinfo{author}{A.~van~der Vaart}, \bibinfo{author}{H.~van Zanten}, \bibinfo{title}{Information rates of nonparametric {G}aussian process methods}, \bibinfo{journal}{Journal of Machine Learning Research} \bibinfo{volume}{12} (\bibinfo{year}{2011}) \bibinfo{pages}{2095--2119}.
\bibitem[{van~der Vaart and van Zanten(2008)}]{VaartZanten2008}
\bibinfo{author}{A.~W. van~der Vaart}, \bibinfo{author}{J.~H. van Zanten}, \bibinfo{title}{Reproducing kernel {H}ilbert spaces of {G}aussian priors}, volume~\bibinfo{volume}{3} of \text{\bibinfo{series}{IMS Collections}}, \bibinfo{publisher}{Institute of Mathematical Statistics}, \bibinfo{year}{2008}, pp. \bibinfo{pages}{200--222}.
\bibitem[{Wang(2021)}]{Wang2021}
\bibinfo{author}{W.~Wang}, \bibinfo{title}{On the inference of applying {G}aussian process modeling to a deterministic function}, \bibinfo{journal}{Electronic Journal of Statistics} \bibinfo{volume}{15} (\bibinfo{year}{2021}) \bibinfo{pages}{5014--5066}.
\bibitem[{Wang et~al.(2020)Wang, Tuo and Wu}]{WangTuoWu2020}
\bibinfo{author}{W.~Wang}, \bibinfo{author}{R.~Tuo}, \bibinfo{author}{C.~F.~J. Wu}, \bibinfo{title}{On prediction properties of kriging: uniform error bounds and robustness}, \bibinfo{journal}{Journal of the American Statistical Association} \bibinfo{volume}{115} (\bibinfo{year}{2020}) \bibinfo{pages}{920--930}.
\bibitem[{Wendland(2005)}]{Wendland2005}
\bibinfo{author}{H.~Wendland}, \bibinfo{title}{Scattered Data Approximation}, number~\bibinfo{number}{17} in \bibinfo{series}{Cambridge Monographs on Applied and Computational Mathematics}, \bibinfo{publisher}{Cambridge University Press}, \bibinfo{year}{2005}.
\bibitem[{Wynne et~al.(2021)Wynne, Briol and Girolami}]{Wynne2020}
\bibinfo{author}{G.~Wynne}, \bibinfo{author}{F.-X. Briol}, \bibinfo{author}{M.~Girolami}, \bibinfo{title}{Convergence guarantees for {G}aussian process means with misspecified likelihoods and smoothness}, \bibinfo{journal}{Journal of Machine Learning Research} \bibinfo{volume}{22} (\bibinfo{year}{2021}) \bibinfo{pages}{1--40}.

\end{thebibliography}

\end{document}